\documentclass[12pt, final]{amsproc}
\usepackage{amssymb}
\usepackage{amscd}
\usepackage{graphicx}
\usepackage[all]{xy}

\usepackage[cp850]{inputenc}
\usepackage{mathrsfs}
\usepackage[notcite,notref]{showkeys}

\theoremstyle{plain}
\newtheorem{theorem}{Theorem}
\newtheorem{proposition}{Proposition}
\newtheorem{lemma}{Lemma}

\newtheorem{corollary}{Corollary}
\newtheorem{example}{Example}

\theoremstyle{remark}
\newtheorem{definition}{Definition}

\def\PO{\operatorname{PO}}

\def\vp{\varphi}
\def\e{\varepsilon}

\def\M{\mathcal{M}}
\def\N{\mathscr{N}}

\def\N{\mathbb N}
\def\R{\mathbb R}
\def\C{\mathbb C}
\def\S{\mathbb{S}}

\newcommand{\cl}{\mathcal}
\newcommand{\lt}{L^p(\mathcal M,\tau)}

\title{Twisting non-commutative $L^p$ spaces}
\author{F\'elix Cabello S\'anchez}
\address{Departamento de Matem\'{a}ticas, Universidad de Extremadura,
Avenida de Elvas, 06071-Badajoz, Spain}
\email{fcabello@unex.es, http://kolmogorov.unex.es/\~{}fcabello}

\author{Jes\'us M. F. Castillo}
\address{Departamento de Matem\'{a}ticas, Universidad de Extremadura,
Avenida de Elvas, 06071-Badajoz, Spain}
\email{castillo@unex.es}

\author{Stanis\l aw Goldstein}
\address{Faculty of Mathematics and Computer Science, University of \L\'od\'z,
ul. Banacha 22, 90-238 \L\'od\'z, Poland} \email{goldstei@math.uni.lodz.pl}

\author{Jes\'us Su\'arez de la Fuente}
\address{Escuela Polit\'ecnica, Universidad de Extremadura,
Avenida de la Universidad, 10071-C\'aceres, Spain}
\email{jesus@unex.es}

\thanks{The research of the first, second and fourth author has been supported in part by project MTM2013-45643-C2-1-P, Spain.}
\thanks{2010 Mathematics Subject Classification 46L52, 46M15, 46B70}

\begin{document}
\noindent{\footnotesize \boxed{\text{\today}}\\[10pt]}

\begin{abstract}
{\bf The paper makes the first steps into the study of extensions (``twisted sums'') of noncommutative $L^p$-spaces regarded as Banach modules over the underlying von Neumann algebra $\M$. Our approach combines Kalton's description of extensions by centralizers (these are certain maps which are, in general, neither linear nor bounded) with a general principle, due to Rochberg and Weiss, saying that whenever one finds a given Banach space $Y$ as an intermediate space in a (complex) interpolation scale, one automatically gets a self-extension $
0\longrightarrow Y\longrightarrow X\longrightarrow Y \longrightarrow 0.$

For semifinite algebras, considering $L^p=L^p(\M,\tau)$ as an interpolation space between $\M$ and its predual $\M_*$ one arrives at a certain self-extension of $L^p$ that is a kind of noncommutative Kalton-Peck space and carries a natural bimodule structure. Some interesting properties of these spaces are presented.

For general algebras, including those of type III, the interpolation mechanism produces two (rather than one) extensions of one sided modules, one of left-modules and the other of right-modules. Whether or not one may find (nontrivial) self-extensions of bimodules in all cases is left open.}
\end{abstract}

\maketitle

\markboth{Cabello-S\'anchez, Castillo, Goldstein and Su\'arez de la Fuente}{Twisting non-commutative $L_p$ spaces}

\section*{Introduction}
In this paper we make the first steps into the study of extensions of noncommutative $L^p$-spaces. An extension (of $Z$ by $Y$) is a short exact sequence of Banach spaces and (linear, continuous) operators
\begin{equation}\label{s1}
0\longrightarrow Y\longrightarrow X \longrightarrow Z\longrightarrow 0.
\end{equation}
This essentially means that $X$ contains $Y$ as a closed subspace so that the corresponding quotient is (isomorphic to) $Z$.

We believe
that the convenient setting in studying extensions of $L^p$-spaces is not that of Banach spaces, but that of Banach modules over the underlying von Neumann algebra $\M$. Accordingly, one should require the arrows in (\ref{s1}) to be homomorphisms.

In this regard it is remarkable and perhaps a little ironic that, while the study of the module structure of general $L^p$-spaces goes back to its inception, the only papers where one can find some relevant information about extensions, namely \cite{k-jfa} and \cite{k-tams}, deliberately neglected this point.

Let us summarize the main results and explain the organization of the paper.

Section 1 contains some preliminaries. Section 2 deals with the tracial (semifinite) case. It is shown that whenever one has a reasonably ``symmetric'' self-extension of the commutative $L^p$ (the usual Lebesgue space of $p$-integrable functions on the line) one can get a similar self-extension
$$
0\longrightarrow L^p(\M,\tau)\longrightarrow X
\longrightarrow L^p(\M,\tau) \longrightarrow 0
$$
of bimodules over any semifinite von Neumann algebra $\M$, equipped with a trace $\tau$.

Our approach combines Kalton's description of extensions by centralizers (these are certain maps which are, in general, neither linear nor bounded) with a general principle, due to Rochberg and Weiss that we can express by saying that whenever one finds a given Banach space $Y$ as an intermediate space in a (complex) interpolation scale, one automatically gets a self-extension $
0\longrightarrow Y\longrightarrow X\longrightarrow Y \longrightarrow 0.
$

Thus for instance, considering $L^p(\M,\tau)$ as an interpolation space between $\M$ and its predual $\M_*$ one arrives at a certain self-extension of $L^p(\M,\tau)$ that we regard as a kind of noncommutative Kalton-Peck space. Some interesting properties of these spaces are presented.

In Section 3 we leave the tracial setting and we consider $L^p$-spaces over general (but $\sigma$-finite) algebras, including those of type III. In this case the interpolation trick still works but produces two (rather than one) extensions of one sided modules, one of left-modules and the other of right-modules. Whether or not one can find (nontrivial) self-extensions of bimodules in all cases is left open.

\section{Preliminaries}

\subsection{Extensions} Let $A$ be a Banach algebra. A quasi-Banach (left) module over $A$ is a quasi-Banach space $X$ together with a jointly continuous outer multiplication $A\times X\to X$ satisfying the traditional algebraic requirements.

 An extension of $Z$ by $Y$ is a short exact sequence of quasi-Banach modules and homomorphisms
\begin{equation}\label{s}
0\longrightarrow Y\stackrel{\imath}\longrightarrow X
\stackrel{\pi}\longrightarrow Z\longrightarrow 0.
\end{equation}
The open
mapping theorem guarantees that $\imath$ embeds  $Y$ as a closed
submodule of $X$ in such a way that the corresponding quotient is
isomorphic to $Z$. Two extensions $0 \to Y \to X_i \to Z \to 0$
($i=1,2$) are said to be equivalent if there exists a homomorphism
$u$ making commutative the diagram $$
\begin{CD}
 0@>>> Y@>>> X_1@>>> Z@>>> 0\\
 & & @| @VV u V @|\\
 0@>>> Y@>>> X_2@>>> Z@>>> 0\\
\end{CD} $$
By the five-lemma \cite[Lemma 1.1]{hs}, and the open
mapping theorem, $u$ must be an isomorphism. We say that (\ref{s})
 splits if it is equivalent to the trivial sequence $0\to Y \to Y
\oplus Z \to Z \to 0$. This just means that $Y$ is a complemented
submodule of $X$, that is, there is a homomorphism $X\to Y$ which
is a left inverse for the inclusion $Y\to X$; equivalently, there
is a homomorphism $Z\to X$ which is a right inverse for the
quotient $X\to Z$.

Operators and homomorphisms are assumed to be continuous. Otherwise we speak of linear maps and ``morphisms''.

Taking $A=\C$ one recovers extensions in the Banach space setting.

Every extension of (quasi-) Banach modules is also an extension of (quasi-) Banach spaces. Clearly, if an extension of modules is trivial, then so is the underlying extension of (quasi-) Banach spaces. Simple examples show that the converse is not true in general. A Banach algebra $A$ is amenable if every extension of Banach modules (\ref{s}) in which $Y$ is a dual module splits as long as it splits an an extension of Banach spaces. This is not the original definition but  an equivalent condition. The original definition reads as follows: $A$ is amenable if every continuous derivation from $A$ into a dual bimodule is inner. Here ``derivation'' means ``operator satisfying Leibniz's rule'' and has nothing to do with the derivations appearing in Section~\ref{cits}.

Every Banach space is a quasi-Banach space and it is possible that the middle space $X$ in (\ref{s}) is only a quasi-Banach space even if both $Z$ and $Y$ are Banach spaces (see \cite[Section 4]{kaltbook}). This will never occur in this paper, among other things because $X$ will invariably be a quotient of certain Banach space of holomorphic functions. Anyway, Kalton proved in \cite{kalt} that if $Z$ has nontrivial type $p>1$ and $Y$ is a Banach space, then $X$ must be locally convex and so isomorphic to a Banach space. In particular, any quasi-norm giving the topology of $X$ must be equivalent to a norm, and hence to the convex envelope norm. If $Z$ is super-reflexive the  proof is quite simple; see \cite{dummies}.

\subsection{Centralizers and the extensions they induce}\label{they}
Let us introduce the main tool in our study of extensions.

\begin{definition}\label{def:centralizer}
Let $Z$ and $Y$ be quasi-normed modules over the Banach algebra $A$ and let $\tilde Y$ be another module containing $Y$ in the purely algebraic sense.
A centralizer from $Z$ to $Y$ with ambient space $\tilde Y$ is a $\mathbb{C}$-homogeneous mapping 
 $\Omega: Z\to \tilde Y$ having the following properties.
\begin{itemize}
\item[(a)] It is quasi-linear, that is, there is a constant $Q$ so that if $f,g\in Z$, then $\Omega(f+g)-\Omega(f)-\Omega(g)\in Y$ and
$
\|\Omega(f+g)-\Omega(f)-\Omega(g)\|_Y\leq Q(\|f\|_Z+\|g\|_Z).
$

\item[(b)] There is a constant $C$ so that if $a\in A$ and $f\in Z$, then $\Omega(af)-a\Omega(f)\in Y$ and
$
\|\Omega(af)-a\Omega(f)\|_Y\leq C\|a\|_A\|f\|_Z.
$
\end{itemize}
\end{definition}

We denote by $Q[\Omega]$ the least constant for which (a) holds and by $C[\Omega]$ the least constant for which (b) holds.

We now indicate the connection between centralizers and extensions. Let $Z$ and $Y$ be quasi-Banach modules and $\Omega:Z\to\tilde Y$ is a centralizer from $Z$ to $Y$. Then
$$
Y\oplus_\Omega Z=\{(g,f)\in \tilde Y\times Z: g-\Omega f\in Y\}
$$
is a linear subspace of $ \tilde Y\times Z$ and $\|(g,f)\|_\Omega=\|g-\Omega f\|_Y+\|f\|_Z$ is a quasi-norm on it (here is the only point where the assumption about the homogeneity of $\Omega$ is used).
Moreover, the map $\imath: Y\to Y\oplus_\Omega Z$ sending $g$ to $(g,0)$ preserves the quasi-norm, while the map $\pi: Y\oplus_\Omega Z\to Z$ given as $\pi(g,f)=f$ is open, so that we have a short exact sequence of quasi-normed spaces
\begin{equation}\label{s2}
0\longrightarrow Y\stackrel{\imath}\longrightarrow Y\oplus_\Omega Z
\stackrel{\pi}\longrightarrow Z\longrightarrow 0
\end{equation}
with relatively open maps. This already implies that $Y\oplus_\Omega Z$ is complete, i.e., a quasi-Banach space.
Actually only the quasi-linearity of $\Omega$ is necessary here.
The estimate in (b) implies that the multiplication $a(g,f)=(ag,af)$ makes  $
Y\oplus_\Omega Z$ into a quasi-Banach module over $A$ in such a way that the arrows in
(\ref{s2}) become homomorphisms. Indeed,
$$
\|a(g,f)\|_\Omega= \|ag-\Omega(af)\|_Y+\|af\|_Z
=  \|ag-a\Omega f+ a\Omega f-\Omega(af)\|_Y+\|af\|_Z
\leq M \|a\|_A\|(g,f)\|_\Omega.
$$
We will always refer to Diagram~\ref{s2} as the extension (of $Z$ by $Y$) induced by $\Omega$.

It is easily seen that two centralizers $\Omega$ and $\Phi$ (acting between the same sets, say $Z$ and $\tilde Y$) induce equivalent extensions if and only if there is a morphism $h:Z\to \tilde Y$ such that $\|\Omega(f)-\Phi(f)-h(f)\|_Y\leq K\|f\|_Z$. If the preceding inequality holds for $h=0$  we say that $\Omega$ and $\Phi$ are equivalent and we write $\Omega\approx\Phi$. In particular $\Omega$ induces a trivial extension if and only if  $\|\Omega(f)-h(f)\|_Y\leq K\|f\|_Z$ for some morphism $h:Z\to \tilde Y$. In this case we say that $\Omega$ is a trivial centralizer.

The corresponding definitions for right modules and bimodules are obvious. Thus, for instance, we define bicentralizers from $Z$ to $Y$ (which are now assumed to be Banach bimodules over the Banach algebra $A$) by requiring $\tilde Y$ to be also a bimodule and replacing the estimate in Definition~\ref{def:centralizer}(b) by
$$
\|\Omega(azb)-a\Omega(z)b\|_Y\leq C\|a\|_A\|f\|_Z\|b\|_A\quad\quad(a,b\in A, z\in Z).
$$

We insist that we are interested in the case of Banach spaces here, so one can assume $Z$ and $Y$ to be Banach spaces. However, the Ribe function $\|\cdot\|_\Omega$ will be only a quasi-norm on $Y\oplus_\Omega Z$, even if it is equivalent to a true norm. See the paragraph closing Section 1.1 and \cite[Appendix 1.9]{cg}.

\subsection{Push-outs and extensions}
The push-out construction appears naturally when one considers two operators defined on the same space. Given operators $\alpha:Y\to A$ and $\beta:Y\to B$, the associated push-out diagram is
\begin{equation}\label{po-dia}
\begin{CD}
Y@>\alpha>> A\\
@V \beta VV @VV \beta' V\\
B @> \alpha' >> \PO
\end{CD}
\end{equation}
Here, the push-out space $\PO=\PO(\alpha,\beta)$ is the quotient of the direct sum $A\oplus B$ (with the sum norm, say) by $S$, the closure of the subspace $\{(\alpha y,-\beta y): y\in Y\}$. The map $\alpha'$ is given by the inclusion of $B$ into $A\oplus B$ followed by the natural quotient map $A\oplus B\to (A\oplus B)/S$, so that $\alpha'(b)=(0,b)+S$ and, analogously, $\beta'(a)=(a,0)+S$.

The diagram (\ref{po-dia}) is commutative: $\beta'\alpha=\alpha'\beta$. Moreover, it is ``minimal'' in the sense of having the following universal property: if $\beta'':A\to C$ and $\alpha'':B\to C$ are operators such that $\beta''\alpha=\alpha''\beta$, then there is a unique operator $\gamma:\PO\to C$ such that $\alpha''=\gamma\alpha'$ and $\beta''=\gamma\beta'$. Clearly, $\gamma((a, b) + S) = \beta''(a)+\alpha''(b)$ and one has $\|\gamma\|\leq \max \{\|\alpha''\|, \|\beta''\|\}$.

Suppose we are given an extension (\ref{s}) and an operator $t: Y\to B$. Consider the push-out of the couple $(\imath, t)$ and draw the corresponding arrows:
\begin{equation*}
\begin{CD}
0  @>>>  Y @>\imath>> X @>>> Z @>>>0 \\
  &  &   @Vt VV  @VVt'V
   & \\
& &B @>\imath'>> \PO  \\
\end{CD}
\end{equation*}
Clearly, $\imath'$ is an isomorphic embedding.
Now, the operator $\pi:X\to Z$ and the null operator $n:B\to Z$ satisfy the identity $\pi\imath=nt=0$, and the universal property of push-outs gives a unique operator $\varpi: \PO\to Z$ making the following diagram commutative:
\begin{equation}\label{po-seq}
\begin{CD}
0  @>>>  Y @>\imath>> X @>\pi>> Z @>>>0 \\
  &  &   @Vt VV  @VVt'V
   @| \\
0  @>>> B @>\imath'>> \PO  @>\varpi >> Z@>>> 0
\end{CD}
\end{equation}
Or else, just take $\varpi((x,b)+S)=\pi(x)$, check
commutativity, and discard everything but the definition of $\PO$.
Elementary considerations show that the lower sequence in the preceding diagram is exact.
That sequence will we referred to as the push-out sequence.
The universal property  of push-out diagrams yields:

\begin{lemma}\label{crit-po} With the above notations,  the push-out sequence splits if and only if $t$ extends to $X$, that is, there is an operator $T:X\to B$ such that $T\imath=t$.\hfill$\square$
\end{lemma}

\subsection{Complex interpolation and twisted sums}\label{cits}
These lines explain the main connection between interpolation and twisted sums we use throughout the paper. General references are  \cite{rochberg-weiss, CJRW, kaltonmontgomery, k-tams, carro}. Let $(X_0,X_1)$ be a compatible couple of complex Banach spaces. This means that both $X_0$ and $X_1$ are embedded into a third topological vector space $W$ and so it makes sense to consider its sum $\Sigma=X_0+X_1=\{w\in W: w=x_0+x_1\}$ which we furnish with the norm $\|w\|_\Sigma=\inf\{\|x_0\|_0+\|x_1\|: w=x_0+x_1\}$ as well as the intersection $\Delta=X_0\cap X_1$ with the norm $\|x\|_\Delta=\max\{\|x\|_0,\|x\|\}$.
We attach a certain space of analytic functions to $(X_0,X_1)$ as follows.

Let $\mathbb S$ denote the closed strip $\mathbb S=\{z\in \mathbb C: 0\leq \Re z\leq 1\}$, and $\mathbb S^\circ$ its interior.
We denote by $\cl{G}=\cl G(X_0,X_1)$  the space of functions $g:\mathbb S\to \Sigma$ satisfying the following
conditions:
\begin{enumerate}
\item $g$ is $\|\cdot\|_\Sigma$-bounded ;

\item $g$ is $\|\cdot\|_\Sigma$-continuous on $\mathbb S$ and
$\|\cdot\|_\Sigma$-analytic on $\mathbb S^\circ$;

\item $g(it)\in X_0, g(it+1)\in X_1$ for each $t\in\R$;

\item the map $t\mapsto g(it)$ is $\|\cdot\|_0$-bounded and
$\|\cdot\|_0$-continuous on $\R$;

\item the map $t\mapsto g(it+1)$ is $\|\cdot\|_1$-bounded and
$\|\cdot\|_1$-continuous on $\R$.
\end{enumerate}

Then $\cl{G}$ is a Banach space under the norm
$
\|g\|_\mathcal G= \sup\{\|g(j+it)\|_j: j=0,1; t\in\R   \}.
$

For $\theta\in[0,1]$, define the interpolation space
$
X_\theta=[X_0,X_1]_\theta=\{x\in\Sigma: x=g(\theta) \text{ for some } g\in\cl G\}
$
with the norm $\|x\|_\theta=\inf\{\|g\|_\cl G: x=g(\theta)\}$. We remark that $
[X_0,X_1]_\theta$ is the quotient of $\cl G$ by $\ker\delta_\theta$, the closed subspace of functions vanishing at $\theta$, and so it is a Banach space.

Now, the basic result is the following.

\begin{lemma}\label{mechanism}
With the above notations, the derivative $\delta'_\theta: \mathcal G\to \Sigma$ is bounded from $\ker \delta_\theta$ onto $X_\theta$ for $0<\theta<1$.
\end{lemma}

\begin{proof}
For a fixed $\theta\in]0,1[$,
let $\varphi$ be a conformal map of $\mathbb S^\circ$ onto the open unit disc sending $\theta$ to $0$, for instance
that given by
\begin{equation}\label{thatgiven}
\varphi(z)=\frac{\exp(i\pi z)-\exp(i\pi\theta)}
{\exp(i\pi z)-\exp(-i\pi\theta)}\quad \text{ for } z\in \mathbb S.
\end{equation}
If $g\in\mathcal G$ vanishes at $\theta$, then one has $g=\varphi h$, with $h\in\mathcal G$ and $\|h\|_\mathcal G= \|g\|_\mathcal G$. Therefore, $g'(\theta)=\varphi'(\theta)h(\theta)$, so $g'(\theta)\in X_\theta$ and
$$
\|g'(\theta)\|_{X_\theta}=
|\varphi'(\theta)|\|h(\theta)\|_{X_\theta}\leq
|\varphi'(\theta)|\|h\|_\mathcal G=
|\varphi'(\theta)|\|g\|_\mathcal G.
$$
Hence $\|\delta_\theta':\ker \delta_\theta\to X_\theta\|\leq |\varphi'(\theta)|$. Notice that $|\varphi'(\theta)|= \pi/(2\sin(\pi\theta))$ when $\varphi$ is given by (\ref{thatgiven}).

Let us see that $\delta_\theta'$ maps $\ker\delta_\theta$ onto $X_\theta$. Take $x\in X_\theta$, with $\|x\|_\theta=1$ and choose $g\in \cl G$ so that $g(\theta)=x$, with $\|g\|_\cl G\leq 1+\epsilon$. Then $h=\varphi g$ belongs to $\ker\delta_\theta$ and $h'(\theta)=\varphi'(\theta)x$.
\end{proof}

In this way, for each $\theta\in]0,1[$ we have a push-out diagram
\begin{equation}\label{basicpo}
\begin{CD}
\ker\delta_\theta @>>> \mathcal G @>\delta_\theta >>X_\theta\\
@V\delta'_\theta VV @VVV @| \\
X_\theta@>>> \PO @>>> X_\theta
\end{CD}
\end{equation}
whose lower row is a self extension of $X_\theta$. The derivation associated with the preceding diagram is the map $\Omega:X_\theta\to\Sigma$ obtained as follows: given $x\in X_\theta$ we choose $g=g_{x}\in\mathcal G$ (homogeneously) such that $x=g(\theta)$ and $\|g\|_\mathcal G\leq (1+\epsilon)\|x\|_{X_\theta}$ for small $\epsilon >0$ and we set $\Omega(x)=g'(\theta)\in\Sigma$. (Note that $\Omega(x)$ lies in $X_\theta$ at least for $x\in\Delta=X_0\cap X_1$.) Homogeneously means that if $g$ is the function attached to $x$ and $\lambda$ is a complex number, then the function attached to $\lambda x$ is $\lambda g$ -- this makes $\Omega:X_\theta\to \Sigma$ homogeneous.

Needless to say, the map $\Omega$ depends on the choice of $g$. However, if $\tilde\Omega(x)$ is obtained as the derivative (at $\theta$) of another $\tilde g\in \mathcal G$ such that $\tilde g(\theta)=x$ and $\|\tilde g\|_\mathcal G\leq M\|x\|$, then $\tilde g-g$ vanishes at $\theta$, so (by Lemma~\ref{mechanism})
$$
\|\tilde\Omega(x)- \Omega(x)\|_{X_\theta}=\|\delta'_\theta(\tilde g-g)\|_{X_\theta}\leq \|\delta_\theta':\ker \delta_\theta\to X_\theta\| (M+1+\epsilon)\|x\|_{X_\theta},
$$
and thus $\tilde\Omega\approx\Omega$.

\begin{lemma}\label{justdefined}
The just defined map $\Omega$ is quasi-linear on $X_\theta$. The extension induced by  $\Omega$ is (equivalent to) the push-out sequence in (\ref{basicpo}).
\end{lemma}

\begin{proof}  That $\Omega$ is quasi-linear is straightforward from Lemma~\ref{mechanism}.

As for the second part, look at the basic Diagram~\ref{basicpo}.
Consider the map $(\delta_\theta',\delta_\theta):\cl G\to X_\theta\oplus_\Omega X_\theta$ given by $(\delta_\theta',\delta_\theta)(f)=(f'(\theta), f(\theta))$. Notice that $(f'(\theta), f(\theta))$ belongs to $X_\theta\oplus_\Omega X_\theta$ for every $f\in\cl G$. Indeed, letting $x=f(\theta)$ we have $f'(\theta)-\Omega(f(\theta))=\delta_\theta'(f-g_x)\in X_\theta$. Moreover,
$$
\|(f'(\theta), f(\theta))\|_\Omega=\|  \delta_\theta'(f-g_x)  \|_\theta+\|f(\theta)\|_\theta\leq M\|f\|_\cl G.
$$
There is an obvious map $\imath: X_\theta\to X_\theta\oplus_\Omega X_\theta$ sending $x$ to $(x,0)$. If $f\in\ker\delta_\theta$ one has
$$
(\delta_\theta',\delta_\theta)(f)=(f'(\theta),0)=\imath\delta_\theta'(f)
$$
and the universal property of the push-out construction yields an operator $u$ making commutative the following diagram
$$
\xymatrix{
\ker\delta_\theta \ar[r] \ar[d]_{\delta'_\theta} & \mathcal G \ar[r]^{\delta_\theta} \ar[d] & X_\theta \ar@{=}[d]\\
X_\theta\ar[r] \ar@{=}[d] & \PO \ar[d] \ar[r] & X_\theta \ar@{=}[d]\\
X_\theta \ar[r] & X_\theta\oplus_\Omega X_\theta \ar[r]&  X_\theta\\
}
$$
This completes the proof.
\end{proof}

The preceding argument is closely related to the observation, due to Rochberg and Weiss \cite{rochberg-weiss}, that
$
X_\theta\oplus_\Omega X_\theta= \cl G/(\ker\delta_\theta\cap \ker\delta_\theta')=\{(f'(\theta), f(\theta)): f\in\cl G\},
$
where the third space carries the obvious (infimum) norm.

An  important feature of the derivation process is that if we start with a couple $(X_0,X_1)$ of Banach modules over an algebra $A$ (this terminology should be self-explanatory by now), then the diagram (\ref{basicpo}) lives in the category of Banach modules and $\Omega$ is a centralizer over $A$.

\section{The tracial (semifinite) case}

\subsection{Some special properties of centralizers on $L^p(\R^+)$}
In this Section we introduce the spaces of measurable functions we shall use along the paper.
Our default measure space is the half line $\R^+=(0,\infty)$. We write $\mathscr B$ for the algebra of Borel sets of $\R^+$ and we denote by $\lambda$ the Lebesgue measure on $\mathscr B$.

Let $L^0$ be the space of all (Borel) measurable functions $f:\R^+\to \C$ equipped with the topology of convergence in measure on sets of finite  measure.
Here we apply the usual convention of identifying functions agreeing almost everywhere. According to Lindenstrauss and Tzafriri \cite[Definition 1.b.17, p. 28]{LindTzaf}, a  K\"othe space on $\R^+$ is a linear subspace $X$ of $L^0$ consisting of locally integrable functions, equipped with a monotone norm (if $f\in X$ and $|g|\leq |f|$ almost everywhere, then $g\in X$ and $\|g\|_X\leq \|f\|_X$) rendering it complete and containing the characteristic function of each Borel set of finite measure. A symmetric space is a K\"othe space $X$ satisfying:
\begin{itemize}
\item If $|f|$ and $|g|$ have the same distribution and $f\in X$, then $g\in X$ and $\|g\|_X=\|f\|_X$.
\item The Fatou property: if $(f_n)$ is an increasing sequence of nonnegative functions of $X$ converging almost everywhere to $f$ and $\sup_n\|f_n\|_X<\infty$, then $f\in X$ and $\|f\|_X=\lim_n\|f_n\|_X$.
\end{itemize}

Of course, if $u$ is a measure-preserving automorphism of $\mathbb R^+$, then the mapping $f\mapsto u^\circ(f)=f\circ u$ defines an isometry on every symmetric space.
We have included the Fatou property in the definition to avoid any difficulty when dealing with spaces of operators. The present definition guarantees that our symmetric spaces are both ``fully symmetric'' (in the sense of \cite{dodds}) and ``rearrangement invariant'' in the sense of \cite{LindTzaf} and \cite{garling}; anyway see cite \cite{kalton-sukochev} for a discussion and related results. If $X$ is a symmetric space, then $L^\infty\cap L^1\subset X\subset L^\infty+L^1$ and the inclusion are continuous; see \cite[Theorem 7.4.2]{garling} for a proof.

It is clear from the definition that every K\"othe space $X$ is an $L^\infty$-module under ``pointwise'' multiplication which turns out to be a submodule of $L^0$.
Let $\Phi: X\to L^0$ be an $L^\infty$-centralizer on $X$. Then $\Phi$ is said to be:
\begin{itemize}
\item Real if it takes real functions to real functions.
\item Symmetric if ($X$ is symmetric and) there is a constant $S$ so that, whenever $u$ is a measure-preserving automorphism of $\R^+$ one has $\|\Phi(u^\circ f)- u^\circ (\Phi f)\|_X\leq S\|f\|_X$.
\item Lazy if, whenever $\mathscr A$ is a $\sigma$-subalgebra of $\mathscr B$ and $f\in X$ is $\mathscr A$-measurable, $\Phi(f)$ is  $\mathscr A$-measurable.
\end{itemize}

Observe that $\Phi$ is lazy if and only if, for every $f\in X$, the function $\Phi(f)$ is measurable with respect to the $\sigma$-algebra generated by $f$, namely $\mathscr A(f)=\{f^{-1}(A): A\text{ is a Borel subset of } \mathbb C\}$. Also, if $\Phi$ is lazy and  $f=\sum_{k=1}^\infty t_k1_{A_k}$ is a function taking only countably many values ($\sigma$-simple from now on), one has $\Phi(f)= \sum_{k=1}^\infty s_k1_{A_k}$ for certain sequence of scalars $(s_k)$.
\medskip

Important examples of  centralizers are given as follows (see \cite{k-memoir}, Section 3 and specially Theorem 3.1). Let $\varphi:\mathbb R^2\to\mathbb \C$ be a Lipschitz function. Then the map $L^p\to L^0$ given  by
\begin{equation}\label{important}
f\longmapsto f\varphi\left(\log\frac{|f|}{\|f\|_p}, \log r_{f}\right).
\end{equation}
is a (symmetric) centralizer on $L^p$ which is real when $\vp$ is real-valued. Here $r_f$ is the so called rank-function of $f\in L^0$ defined by
$$
r_f(t)=\lambda\{s\in\R^+: |f(s)|>|f(t)| \text{ or } s\leq t \text{ and } |f(s)|=|f(t)|\},
$$
which arises in real interpolation (cf. \cite{jrw}).

For what this paper is concerned, the crucial result on $L^\infty$-centralizers is the following.

\begin{theorem}[Kalton \cite{k-tams}, Theorem 7.6]\label{crucial}
There is a (finite) constant $K$ so that whenever
$1 < p \leq2$ and $X$ is a $p$-convex and $q$-concave K\"othe function space with
 $p^{-1} +q^{-1}=1$ and $\Phi$ is a real centralizer on $X$ with $C[\Phi] < 200/q$ then there is a
pair of K\"othe function spaces $(X_0, X_1)$ so that $X = [X_0, X_1]_{1/2}$ (with equivalent norms) and if $\Omega:X\to L^0$ is the corresponding derivation,
then $\|\Phi(f)-\Omega(f)\| \leq K\|f\|$ for $f\in X$. In particular $\Phi\approx\Omega$.

If $\Phi$ is symmetric, then $X_0$ and $X_1$ can be taken to be symmetric.
\end{theorem}

Before going any further, let us see some useful consequences.

\begin{lemma}\label{anyfurther} Let $p\in(1,\infty)$.
\begin{itemize}
\item[(a)] Every centralizer on $L^p$ is equivalent to a linear combination of two derivations.
\item[(b)] Every symmetric centralizer on $L^p$ is equivalent to a lazy centralizer.
\item[(c)] Every symmetric centralizer on $L^p$ takes values in $L^1+L^\infty$.
\end{itemize}
\end{lemma}

\begin{proof}
(a) It is obvious from Theorem~\ref{crucial} that if $\Phi$ is a real centralizer on $L^p$ then $c\Phi$ is equivalent to a derivation for $c>0$ sufficiently small, hence $\Phi$ is equivalent to a constant multiple of a derivation.
If $\Phi$ is any (symmetric) centralizer on $L^p$, then
letting $\Phi_1(f)=\Re\Phi(\Re(f))+i\Re\Phi(\Im(f))$ and $\Phi_2(f)=\Im\Phi(\Re(f))-i\Im\Phi(\Im(f))$ one has
$\Phi\approx\Phi_1+i\Phi_2$ with $\Phi_1$ and $\Phi_2$ real (symmetric) centralizers and the result follows.

(b)
Let $\Phi$ be a symmetric centralizer on $L^p$, where $1<p<\infty$. We shall prove that $\Phi$ ``almost  commutes'' with every conditional expectation operator in the following sense: there is a constant $C$ such that for every $\sigma$-algebra $\mathscr A\subset \mathscr B$ and every $f\in L^p$, one has
\begin{equation}\label{EA}
\|E^\mathscr A\Phi f- \Phi(E^\mathscr A(f))\|_{L^p}\leq C\|f\|_{L^p},
\end{equation}
where $E^\mathscr A$ is the conditional expectation operator; see \cite[p.~122]{LindTzaf} for the definition.
After that the result follows just considering the mapping $f\mapsto E^{\mathscr A(f)}(\Phi f)$ which gives a (necessarily symmetric) lazy centralizer equivalent to $\Phi$.
By (a) we may assume that $\Phi$ is a derivation, so that there are a couple of symmetric spaces $X_0, X_1$ so that $[X_0, X_1]_{1/2}=L^p$ with equivalent norms and $\Phi(f)=G_f'(\tfrac{1}{2})$, where $G_f\in\mathcal G(X_0,X_1)$ is such that $G_f(\tfrac{1}{2})=f$ and   $\|G_f\|_{\mathcal G(X_0,X_1)}\leq M\|f\|_{L^p}$ for some constant $M$ independent on $f$. Since $L^\infty+L^1$ contains both $X_0$ and $X_1$, it also contains its sum, so $\Phi(f)\in L^\infty+L^1$  and $E^\mathscr A\Phi f$ is correctly defined.

On the other hand, if $\mathscr A\subset\mathscr B$ is a $\sigma$-algebra, $E^\mathscr A$ is a contractive projection on every symmetric space (see \cite[Theorem 2.a.4]{LindTzaf}), hence if  $g\in\mathcal G(X_0,X_1)$, then $E^\mathscr A\circ g$ also belongs to  $\mathcal G(X_0,X_1)$ and $\| E^\mathscr A\circ g\|\leq \|g\|$.

Now, if $f\in L^p$ and $\mathscr A\subset\mathscr B$ is a $\sigma$-algebra, letting $h=E^\mathscr A(f)$
we consider the functions $G_f$ and $G_h$. Then $E^\mathscr A\circ G_f-G_h$ vanishes at $z=\tfrac{1}{2}$, so
$$
\|E^\mathscr A\Phi f- \Phi(E^\mathscr A(f))\|_{L^p}= \|\delta_{1/2}'\left( E^\mathscr A\circ G_f-G_h \right)\|\leq M \|\delta_{1/2}'\|\|f\|_{L^p}.
$$
This proves (b) and (c) for derivations and the general case follows from (a).
\end{proof}

\subsection{From commutative to noncommutative}
This Section depends on the Spectral Theorem that we now recall, mainly to fix notations. The reader is referred to \cite[Section~VIII.3]{reed-simon} for a complete exposition. Let $\mathcal H$ be a Hilbert space. A closed and densely defined operator $x:D(x)\to\mathcal H$ is self-adjoint when $D(x)=D(x^*)$ and $x^*=x$.

For every self-adjoint $x$ there exists a unique ``spectral measure'' $e^x:\mathscr B(\mathbb R)\to B(\mathcal H)$ (this means that $e^x(B)$ is an orthogonal projection for each Borel $B$ and that $e^x(\cdot)$ is $\sigma$-additive with respect to the strong operator topology of $B(\mathcal H)$ such that
$$
x=\int_\mathbb R \lambda de^x(\lambda).
$$
If $x$ is a closed, densely defined operator, then $x^*x$ is self-adjoint (and, actually, positive). The modulus of $x$ is then defined as
$$
|x|=(x^*x)^{1/2}=\int_{\mathbb R^+} \lambda^{1/2} de^{(x^*x)}(\lambda).
$$
One has the ``polar decomposition'' $x=u|x|$, where $u$ is a partial isometry which is ofted called the phase of $x$.

Let $\M$ be a semifinite von Neumann algebra with a faithful, normal, semifinite (fns) trace
$\tau$, acting on $\mathcal H$. A closed
densely defined operator on $\mathcal H$ is affiliated with $\M$ if its
spectral projections (that is, the projections $e^{(x^*x)}(B)$ for $B\in\mathscr B(\mathbb R)$) belong to $\M$. A closed, densely defined
operator $x$ affiliated with $\M$ is called
$\tau$-measurable if, for any $\epsilon>0$, there exists a
projection $e\in \M$ such that $e\mathcal H\subset D(x)$ and
$\tau(1-e)\leq\epsilon$. We denote the set of all
$\tau$-measurable operators affiliated with a von Neumann algebra
$\M$ by $\widetilde{\M}$. The so called
measure topology on $\widetilde{\M}$ is the least linear topology containing the sets
$$\{x\in\widetilde{\M}:
\text{there exists a projection } e\in \M \text{ such that
}\tau(1-e)<\varepsilon, xe\in \M \text{ and }\|xe\|<\varepsilon\},$$
with $\varepsilon>0$.
Endowed with measure topology, strong sum, strong product and
adjoint operation as involution, $\widetilde{\M}$ becomes a
topological *-algebra (see \cite{nelson, conradie} for basic information). The trace $\tau$ has a natural extension to
$\widetilde{\M}_+$.

We define $L^p(\M,\tau)$ as the space of all $\tau$-measurable
operators $x$ such that $\tau(|x|^p)<\infty$,
with norm $\|x\|_p=\left(\tau(|x|^p)\right)^{1/p}$.

More general spaces of operators can be introduced as follows \cite{dodds, deP,kalton-sukochev}.
Let $x$ be a measurable operator, so that $\tau \left(e^{|x|}(\lambda,\infty)\right)$ is finite for some $\lambda >0$. The generalized singular value function of $x$ is the function $\mu(x):\R^+\to[0,\infty]$ given by
$$
\mu(x)(t)=\inf\{\lambda>0:\tau \left(e^{|x|}(\lambda,\infty)\right)\leq t\}.
$$
Now, if $X$ is a symmetric function space, the corresponding ``symmetric operator space'' is
$$
X(\M,\tau)=\{x\in \widetilde\M: \mu(x)\in X\}, \quad\text{with}\quad \|x\|=\|\mu(x)\|_X.
$$
An important feature of these spaces is that they are bimodules over $\M$ with the obvious outer multiplications.

In order to state the main result of the Section, let us consider a self-adjoint $y\in \widetilde\M$ and let $\M_y$ be the (von Neumann) subalgebra of $\M$ generated by the spectral projections of $y$. By general representation results one can construct a *-homomorphism $\xi:\M_y\to L^\infty$ preserving the trace, that is, such that $\tau(a)=\int_0^\infty\xi(a)d\lambda$ for every nonnegative $a\in \M_y$.
A simple proof of this fact appears in \cite[Proof of Theorem 2.1]{pisier-handbook} (note that ``our'' $\xi$ is the inverse of the map that Pisier and Xu call $S$). A different proof for finite (von Neumann) algebras can be seen in \cite[Theorem 3.2.5]{sinclair} (the argument works for semifinite algebras as well). For a more general result, see \cite[Theorem 3.5]{dodds}.

If $\widetilde\M_y$ denotes the closure of $\M_y$ in $\widetilde\M$, then $\xi$ extends to a continuous *-homomorphism $\widetilde \M_y\to L^0$ that we denote again by $\xi$. Clearly, $\xi(\M_y)=L^\infty(\mathbb R^+,\mathscr A,\lambda)$, where $\mathscr A$ is a $\sigma$-subalgebra of $\mathscr B$. It follows that for every $\mathscr A$-measurable $f\in L^1$ there is $z\in
\widetilde\M_y$ (actually in $L^1(\M,\tau)$) such that $f=\xi(z)$ and so $\xi^{-1}(f)$ is correctly defined if $f\in L^1+L^\infty$ is $\mathscr A$-measurable. Besides, $\xi$ preserves every ``symmetric'' norm in the following sense: if $X$ is a symmetric function space on $\mathbb R^+$ and $f\in X$ is $\mathscr A$-measurable, then there is $x\in \widetilde \M_y$ such that $\xi(x)=f$ and $\|x\|_{X(\mathcal M,\tau)}=\|f\|_X$. This is obvious since $\mu(x)$ and $f$ have the same distribution.\medskip

The following result and its proof are modeled on \cite[Theorem 8.3]{k-tams}:

\begin{theorem}\label{main}

Let $\Phi$ be a lazy, symmetric $L^\infty$-centralizer on $L^p$, where $1<p<\infty$. Given a semifinite von Neumann algebra $(\M,\tau)$
we define a mapping $\Phi_\tau: L^p(\M, \tau)\to\widetilde\M$ as follows: For each $x\in L^p(\M,\tau)$  we choose a trace preserving *-homomorphism $\xi:  \widetilde \M_{|x|}\to L^0$ (depending only on $\M_{|x|}$)  as before and we set
\begin{equation}\label{Pt}
\Phi_\tau(x)=u\cdot \xi^{-1}(\Phi(\xi(|x|))),
\end{equation}
where  $x=u|x|$ is the polar decomposition. $\Phi_\tau$ is an $\M$-bicentralizer on $L^p(\M, \tau)$ and all mappings defined in this way are equivalent, independently of the choice of
$\xi$.\end{theorem}

\begin{proof}
First of all observe that the definition of $\Phi_\tau$ makes sense since $\Phi(\xi(|x|))$ belongs to $L^\infty+L^1$ (see Lemma~\ref{anyfurther}) and it is measurable with respect to the $\sigma$-algebra generated by
$\xi(|x|)$ and so $\xi^{-1}\Phi(\xi(|x|))$ is well defined. Also, note that since $\M_{|\lambda x|}=\M_{|x|}$ for each nonzero $\lambda\in\C$ and $\xi$ depends only on its domain algebra the  resulting map $\Phi_\tau$ is homogeneous.

Let us prove that $\Phi_\tau$ is a bicentralizer assuming that $\Phi$ is a derivation. Precisely, we are assuming there is a couple of symmetric K\"othe spaces on $\R^+$ such that $L^p=[X_0,X_1]_{1/2}$, with equivalent norms in such a way that, for each $f\in L^p$ one has $\Phi(f)=g'({1\over 2})$, where $g\in \mathcal G(X_0,X_1)$ satisfies $g({1\over 2})=f$ and $\|g\|_\mathcal G\leq K\|f\|_p$.

Set $X=[X_0,X_1]_{1/2}$, with the natural norm. This is a symmetric  K\"othe space on $\R^+$. The key point is that the formula
\begin{equation}\label{taking}
[X_0(\M, \tau),X_1(\M,\tau)]_{1\over 2}=X(\M,\tau)
\end{equation}
holds for all semifinite algebras $(\M,\tau)$ -- see \cite[Theorem 3.2]{dodds} and \cite{pisier-handbook}.

Of course $X(\M,\tau)=\lt$, up to equivalence of norms and we may consider the corresponding derivation on $\lt$.
That is, given $x\in\lt$ we choose $G_x\in\cl G(X_0(\M,\tau),X_1(\M,\tau))$ such that $G_x({1\over 2})=x$ and $\|G_x\|_\mathcal G\leq (1+\epsilon)\|x\|_{X(\M,\tau)}\leq K\|x\|_p$ and then we put
$$
\Omega(x)=\delta_{1/2}'G_x\in \widetilde\M.
$$
The fact that such an $\Omega$ turns out to be an $\M$-bicentralizer on $\lt$  should be obvious by now, but let us record the proof for future reference. Take $x\in X(\M,\tau)$ and $a,b\in \M$ (that we regard as constant functions on $\mathbb S$). We have $G_{axb}-aG_xb\in\ker\delta_{1/2}$ by the very definition. Moreover,
$$
\|G_{axb}-aG_xb\|_\mathcal G\leq \|G_{axb}\|_\mathcal G+\|aG_xb\|_\mathcal G\leq 2(1+\epsilon)\|a\|\|x\|_{X(\M,\tau)}\|b\|,
$$\
so
$$
\begin{aligned}
\|\Omega(axb)-a\Omega(x)b\|_{X(\M,\tau)}&= \|\delta_{1/2}'(G_{axb}-aG_xb)\|_{X(\M,\tau)}\\
&\leq
\|\delta_{1/2}':\ker \delta_{1/2}\to X(\M,\tau)\|2(1+\epsilon)\|a\|\|x\|_{X(\M, \tau)}\|b\|\\
&\leq (1+\epsilon)\pi \|a\|\|x\|_{X(\M,\tau)}\|b\|.\\
\end{aligned}
$$
Thus, to complete the proof that the formula (\ref{Pt}) defines a bicentralizer on $\lt$ , it suffices to see that one can choose the functions $G_x$ in such a way that $G_x'({1\over 2})=\Phi_\tau$.

So, pick a normalized $x\in \lt$ and put $f=\xi(|x|)$. Then $f$ is normalized in $L^p$ and we have $\Phi(f)=\delta_{1/2}'g$ where $g\in \mathcal G(X_0,X_1)$ is the corresponding extremal--recall that we are assuming that $\Phi$ is itself a derivation.

We claim that the mapping $G:\mathbb S\to \widetilde\M$ given by
$
G(z)=u\cdot\xi^{-1} E^\mathscr A (g(z))
$
is allowable for $x$. We have
$$
G(\tfrac{1}{2})=u\cdot\xi^{-1} E^\mathscr A (g(\tfrac{1}{2}))=u\cdot\xi^{-1} E^\mathscr A (f)=u\xi^{-1}f=u|x|=x.
$$
That $G$ belongs to $\cl G(X_0(\M,\tau),X_1(\M,\tau))$ is obvious since $E^\mathscr A$ is contractive on $X_0$ and $X_1$ (hence on $X_0+X_1$) and $\xi$ preserves all symmetric norms: actually the norm of $G$ in $\cl G(X_0(\M,\tau),X_1(\M,\tau))$ cannot exceed that of $g$ in $\mathcal G(X_0,X_1)$.

Finally, applying  the chain rule and taking into account that $\Phi$ is lazy,
$$
G'(\tfrac{1}{2})=u\cdot\xi^{-1} E^\mathscr A (g'(\tfrac{1}{2}))=u\cdot\xi^{-1} E^\mathscr A (\Phi(f))=u\xi^{-1}\Phi(f)=u\xi^{-1}\Phi(\xi(|x|))=\Phi_\tau(x).
$$
And so $\Phi_\tau$ is a bicentralizer.

To complete the proof (still under the assumption that $\Phi$ is a derivation) we must prove that $\Phi_\tau$ is essentially independent of the family of *-homomorphisms $\xi$. Indeed, if $\xi_1:\widetilde\M_{|x|}\to L^0$ is another trace-preserving *-homomorphism and $\mathscr A_1\subset \mathscr B$ is the corresponding
$\sigma$-algebra, letting $f_1=\xi_1(|x|)$ and taking any allowable $g_1\in \mathcal G(X_0,X_1)$ so that $g_1(\tfrac{1}{2})=f_1$ with $\|g_1\|\leq (1+\varepsilon)\|f_1\|_X= \|x\|_{X(\M,\tau)}$  we have that if
$$
G_1(z)=u\cdot\xi_1^{-1} E^{\mathscr A_1} (g_1(z))\quad\quad(z\in\mathbb S),
$$
then $G_1$
belongs to $\cl G(X_0(\M,\tau),X_1(\M,\tau))$, is allowable for $x$ and $G_1'(\tfrac{1}{2})=u\cdot\xi_1^{-1}\Phi(\xi_1(|x|))$ and since $\delta'_{1/2}$ is bounded from $\ker\delta_{1/2}$ to $X(\M,\tau)$ (see Lemma~\ref{mechanism}) we have that  $G_1'(\tfrac{1}{2})-G'(\tfrac{1}{2})$ falls in $X(\M,\tau)$ and
$$
\|u\xi_1^{-1}\Phi(\xi_1(|x|))- u\xi^{-1}\Phi(\xi(|x|))\|_{X(\M,\tau)}= \|\delta'_{1/2}(G_1-G)\|\leq M\|x\|_{{X(\M,\tau)}}.
$$
This completes the proof when $\Phi$ is a derivation--or a linear combination of derivations.


To finish, observe that if $\Psi$ and $\Phi$ are two equivalent lazy centralizers on $L^p$, then the maps $\Psi_\tau$ and $\Phi_\tau$ are equivalent on $\lt$--at least if the ``prescribed'' family of *-homomorphisms $x\mapsto\xi$ is fixed. Indeed, if $x\in L^p(\M,\tau)$, then
$$
\|\Psi_\tau(x)-\Phi_\tau(x)\|_{L^p(\M,\tau)}=\left\|u\xi^{-1}\left(\Psi(\xi(|x|)-\Phi(\xi(|x|)\right)\right\|\leq \|\Psi-\Phi\|\|(\xi(|x|))\|\leq M\|x\|.
$$
Now, the result follows from Lemma~\ref{anyfurther}.
\end{proof}

The action of $\Phi_\tau$ on $\sigma$-elementary operators if quite simple. Here, a $\sigma$-elementary operator is one of the form $x=\sum_{k=1}^\infty\lambda_k e_k$, with $e_k$ disjoint projections and $\lambda_k\in\mathbb C$.  Indeed, for such an $x$ we have $|x|=\sum_{k=1}^\infty|\lambda_k| e_k$ and $ u=\sum_{k=1}^\infty u_k e_k$, where $u_k$ is the signum of $\lambda_k$. Hence, if $\xi:\widetilde \M_{|x|}\to L^0$ is  any trace-preserving *-homomorphism, then $f=\xi(|x|)= \sum_{k=1}^\infty|\lambda_k| 1_{A_k}$, where $(A_k)$ is a sequence of disjoint Borel sets of $\mathbb R^+$ and $\xi(e_k)=1_{A_k}$ for every $k\in\mathbb N$. Now, as $\Phi$ is lazy, we have $\Phi(f)= \sum_{k=1}^\infty s_k 1_{A_k}$
for some sequence $(s_k)$ and
$$
\Phi_\tau(x)= u\xi^{-1}\left( \sum_{k=1}^\infty s_k 1_{A_k}	\right)= \sum_{k=1}^\infty u_ks_ke_k.
$$

The following result applies to many centralizers appearing in nature. In particular, it applies to the centralizers given by (\ref{important}) when $\vp$ depends only on the first variable, by just taking $\phi(t)=t\vp(\log t)$ for $t\in\R^+$.

\begin{corollary}\label{innature}
Let $\Phi$ be a centralizer on $L^p$, where $1<p<\infty$. Suppose there is a Borel function
 $\phi:\R^+\to\C$ such that $\Phi(f)=\phi\circ f$ for every $f\geq 0$ normalized in $L^p$.
Then, for every semifinite von Neumann algebra $(\M,\tau)$, the map $x\mapsto \|x\|_p u \phi(|x|/\|x\|_p)$ is an $\mathcal M$-bicentralizer on $\lt$.
\end{corollary}

\begin{proof} It is obvious that $\Phi$ is both symmetric and lazy. In view of Theorem~\ref{main} it suffices to check that $\Phi_\tau(x)=\phi(x)$ for $x$ positive and normalized in $\lt$. Let $\xi$ be the prescribed trace-preserving *-homomorphism.
Since $\Phi_\tau(x)=\xi^{-1}\Phi(\xi(x))=\xi^{-1}(\phi\circ(\xi(x))$ the proof will be complete if we show that $\phi\circ(\xi(x)=\xi(\phi(x))$, where $\phi(x)$ is given by the functional calculus:
$$
\phi(x)=\int_0^\infty\phi(\lambda)de^{x}(\lambda), \quad\text{where}\quad x=\int_0^\infty \lambda de^{x}(\lambda).
$$
Let us consider $L^\infty$ as a von Neumann algebra with trace $\lambda$ (to be true, the trace of $a\in L^\infty$ is $\int_{\mathbb R^+}ad\lambda$) acting by multiplication on $L^2$ and let $\widetilde L^\infty$ be the space of $\lambda$-measurable operators (affiliated with $L^\infty$). Set $f=\xi(x)$ which we may now regard also as a self-adjoint operator in $\widetilde L^\infty$. Then, if
$$
f=\int_0^\infty \lambda de^f(\lambda)
$$
is the spectral representation it is obvious that $\xi e^x= e^f$ in the sense that for every $B\in\mathscr B$ one has $\xi(B)= e^f(B)$. Moreover, $e^f(B)$ can be identified with $1_B\circ f= 1_{f^{-1}(B)}$ and so
$$
\xi(\phi(x))=\xi\left(	 \int_0^\infty\phi(\lambda)de^{x}(\lambda) \right)
= \int_0^\infty\phi(\lambda)d\xi e^{x}(\lambda) =
\int_0^\infty\phi(\lambda)de^{f}(\lambda) =\phi\circ f,
$$
and we are done.
\end{proof}

\subsection{Noncommutative Kalton-Peck spaces}\label{kalton-peck}
In this part we discuss the simplest case of self-extensions, namely that one obtains out from the identity $[\M, L^1(\M,\tau)]_\theta=L^p(\M,\tau)$ at $\theta=1/p$. In order to simplify the computation of extremals we introduce a larger space of holomorphic functions as follows.
We consider both $\M$ and $L^1(\M, \tau)$ as subspaces of $\widetilde\M$ and we set $\Delta=\M\cap L^1(\M, \tau)$ and
 $\Sigma=\M+ L^1(\M,\tau)$.
Let $\cl{H}=\cl{H}(\M,\tau)$ be the space of functions $h:\mathbb S\to\Sigma$
satisfying the following conditions:
\begin{enumerate}
\item $h$ is $\|\cdot\|_\Sigma$-bounded.

\item For each $x\in\Delta$ the function $z\longmapsto \tau(xh(z))$ is continuous on
$\mathbb S$ and analytic on $\mathbb S^\circ$.

\item $h(it)\in \M, h(it+1)\in L^1(\tau)$ for each $t\in\R$;

\item the map $t\mapsto h(it)$ is $\|\cdot\|_\infty$-bounded and
$\sigma(\M,L^1(\tau))$-continuous on $\R$;

\item the map $t\mapsto h(it+1)$ is $\|\cdot\|_1$-bounded and
$\|\cdot\|_1$-continuous on $\R$.
\end{enumerate}

We equip $\cl{H}$ with the norm $\|h\|_{\mathcal H}=\sup\{\|h(it)\|_\M, \|h(it+1)\|_1:t\in
\R\}\}$. Note that the elements of $\cl{H}$ are in fact
$\|\cdot\|_\Sigma$-analytic on $\mathbb S^\circ$.

Letting $\theta=1/p\in(0,1)$ we have that $\delta_\theta$ maps $\mathcal H$ onto $L^p(\tau)$ (without increasing the norm) and replacing $\mathcal G$ by $\mathcal H$ everywhere in the proof of Lemma~\ref{mechanism} we see that the restriction of $\delta_\theta'$ to $\ker\delta_\theta$ is a bounded operator onto $L^p(\tau)$ and we can form the push-out diagram
\begin{equation}\label{large}
\begin{CD}
\ker\delta_\theta @>>> \mathcal H @>\delta_\theta >>L^p(\tau)\\
@V\delta'_\theta VV @VVV @| \\
L^p(\tau) @>>> \PO @>>> L^p(\tau)
\end{CD}
\end{equation}
Please note that the above diagram lives in the category of bimodules over $\M$. Also, as $\mathcal H$ contains the Calder\'on space $\mathcal G$ it is really easy to see that this new push-out extension is in fact the same one gets by using $\mathcal G$.

Let us compute the extremals associated to the quotient $\delta_\theta:\mathcal H\to L^p(\tau)$. Suppose $f\in L^p(\tau)$ is a positive operator with $\|f\|_p=1$. It is easily seen that the function
$h(z)=f^{pz}$ belongs to $\mathcal H$ (although it is not in $\mathcal G$ in general) and also that $\|h\|_\mathcal H=1$.
Of course, $h'(\theta)=pf\log f$ and thus, the derivation associated to Diagram~\ref{large} is given by
\begin{equation}\label{op}
\Omega_p(f)=pf\log(|f|/\|f\|_p)\quad\quad(f\in L^p(\tau)).
\end{equation}
Let us denote the corresponding twisted sum $L^p(\tau)\oplus_{\Omega_p}L^p(\tau)$ by $Z_p(\tau)$. Our immediate aim is to prove the following.

\begin{theorem}\label{k-semi}
$Z_p(\M, \tau)$ is a nontrivial self extension of $L^p(\M,\tau)$ as long as $\M$ is infinite dimensional and $1<p < \infty$.
\end{theorem}

\begin{proof}Needless to say $Z_p(\M, \tau)$ is a bimodule extension over $\M$. We shall prove that it doesn't split even as an extension of Banach spaces.
As $\M$ is infinite dimensional there is a sequence $(e_i)$ of mutually orthogonal projections having finite trace. Let $\mathcal A$ be the von Neumann subalgebra of $\M$ spanned by these projections. Notice that we may consider $\widetilde{\mathcal A}$ as a *-subalgebra of $\widetilde\M$ and $L^p(\mathcal A,\tau)$ as a subspace of $\lt$.

Clearly, $\Omega_p$ maps $L^p(\mathcal A, \tau)$ to $\widetilde{\mathcal A}$ as an $\mathcal A$-centralizer and we have a commutative diagram of inclusions
$$
\xymatrix{
L^p(\mathcal A, \tau) \ar[r] \ar[d] &  L^p(\mathcal A, \tau) \oplus_{\Omega_p} L^p(\mathcal A, \tau)  \ar[r] \ar[d] & L^p(\mathcal A, \tau) \ar[d]\\
L^p(\M, \tau)\ar[r] &  Z_p(\M, \tau) \ar[r]&  L^p(\M,\tau)
}
$$
On the other hand, the ``conditional expectation'' given by
$$
E_\mathcal A(f)=\sum_i \frac{\tau(fe_i)}{\tau(e_i)}e_i
$$
is a contractive projection on $L^p(\M, \tau)$ whose range is $ L^p(\mathcal A,\tau)$. The immediate consequence of all this is that
if the lower extension of the preceding  diagram splits, then so does the upper one.

Let us check that this is not the case. As
$\mathcal A$ is amenable (it is isometrically *-isomorphic to the algebra $\ell^\infty$) and $ L^p(\mathcal A, \tau)$ is a dual bimodule (it is isometrically isomorphic to $\ell^p$, which is reflexive) we have that the upper row in the above diagram splits as an extension of Banach spaces if and only if it splits as an extension of Banach $\mathcal A$-modules. And this happens if and only if there is a morphism $\phi:L^p(\mathcal A, \tau) \to \widetilde{\mathcal A}$ approximating $\Omega_p$ in the sense that
\begin{equation}\label{sense}
\|\Omega_p(f)-\phi(f)\|_p\leq \delta\|f\|_p
\end{equation}
for some constant $\delta$ and every $f\in L^p(\mathcal A, \tau)$. It is clear that every morphism $\phi:L^p(\mathcal A, \tau)\to \widetilde{\mathcal A}$ has the form $\phi(\sum_it_ie_i)=\sum_i\phi_i t_i e_i$ for some sequence of complex numbers $(\phi_i)$. Taking $f=e_i$ in (\ref{sense}) we see that $|\phi_i+\log\tau(e_i)|\leq \delta$. It follows that if (\ref{sense}) holds for some $\phi=(\phi_i)$ then it must hold for $\phi_i=-\log\tau(e_i)$, possibly doubling the value of $\delta$.

Fix $n\in\mathbb N$ and take $f=\sum_{i=1}^n t_ie_i$ normalized in $L^p(\tau)$ in such a way that the nonzero summands in the norm of $f$ agree:
$$
f=\sum_{i=1}^n (n\tau(e_i))^{-1/p}e_i.
$$
For this $f$ and $\phi_i=-\log\tau(e_i)$ we have $\Omega_p(f)-\phi(f)=-\log(n)f$, so $\|\Omega_p(f)-\phi(f)\|_p=\log(n)$, which makes impossible the estimate in (\ref{sense}).
\end{proof}


\subsection{Duality}
In this Section we extend Kalton-Peck duality results in \cite{kaltpeck} to all semifinite algebras by showing that for every trace $\tau$ the dual space of $Z_p(\M, \tau)$ is isomorphic to $Z_q(\M,\tau)$, where $p$ and $q$ are conjugate exponents, that is, $p^{-1}+q^{-1}=1$. (see \cite[Theorem 3.2]{newJ} for the particular case of Schatten classes). In order to achieve a sharp adjustment of the parameters, let us agree that, given $p\in(1,\infty)$ and a Lipschitz function $\vp:\R\to\C$, the associated Kalton-Peck centralizer $\Phi_p:L^p(\M,\tau)\to\widetilde\M$ is defined by $\Phi_p(f)=f\vp(p\log(|f|/\|f\|_p)$ and the corresponding Kalton-Peck space is $Z_p^\vp(\M,\tau)=L^p(\tau)\oplus_{\Phi_p} L^p(\tau)$. This is coherent with (\ref{op}), where $\vp$ is the identity on $\R$.

\begin{theorem}\label{th:duality}
Let $p$ and $q$ be conjugate exponents, $\vp$ a Lipschitz function, and $\tau$ be a trace. Then $
Z_q^\vp(\tau)$ is isomorphic to the conjugate of $
Z_p^\vp(\tau)$ under the pairing
\begin{equation}\label{pairing}
\langle(x,y), (v,w)\rangle=\tau(xw-yv)\quad\quad((x,y)\in
Z_q^\vp(\tau), (v,w)\in Z_p^\vp(\tau))
\end{equation}
\end{theorem}

\begin{proof} The proof depends on the following elementary inequality: given $s,t\in\C$ one has
\begin{equation}\label{ineq}
\left| ts\left(\log \frac{\left|t\right|^q}{\left|s\right|^{p}}\right)\right|\leq \frac{p}{e}\left(\left |t \right |^q+ \left |s\right|^p\right).
\end{equation}
This is (a rewording of) the case $n=1$ of \cite[Lemma 5.2]{kaltpeck} that Kalton and Peck use in the proof of \cite[Theorem 5.1]{kaltpeck}.
Let us see that the pairing is continuous. To this end write
$$
xw-yv=(x-\Phi_q(y))w+\Phi_q(y)w-y(v-\Phi_p(w))-y\Phi_p(w).
$$
As $\|(x-\Phi_q(y))w\|_1\leq \|x-\Phi_q\|_q\|w\|_p$ and, similarly,
$\|y(v-\Phi_p(w))\|_1\leq \|y\|_q\|v-\Phi_p(w)\|_p$ it suffices to obtain  an estimate of the form
\begin{equation}\label{suf}
|\tau(\Phi_q(y)w-y\Phi_p(w))|\leq M\|y\|_q\|w\|_p.
\end{equation}

First, let us assume $y$ and $w$ are $\sigma$-elementary operators with $\|y\|_q=\|w\|_p=1$ and representations
$y=\sum t_i y_i$ and $w=\sum s_j w_j$ converging in $L^q(\tau)$ and $L^p(\tau)$, respectively. We may assume with no loss of generality that $\sum_i y_i=\sum_j w_j=\bf 1_\M$ (summation in the $\sigma(\M,\M_*)$ topology).
We have
\begin{align*}
\Phi_q(y)w-y\Phi_p(w)&=\left(\sum_i t_i\vp(q\log|t_i|)y_i \right)\left( \sum_js_jw_j\right)- \left(\sum_i t_i y_i \right)\left( \sum_js_j\varphi(p\log|s_j|)w_j\right)\\
&=\sum_{i,j} t_is_j\left(\varphi(\log|t_i|^q)- \varphi(\log|s_j|^p)\right)y_iw_j.
\end{align*}
Applying (\ref{ineq}) and taking into account that the product of any two projections has positive trace we can estimate the left-hand of (\ref{suf}) as follows:
\begin{eqnarray*}
|\tau(\Phi_q(y)w-y\Phi_pw)|&\leq &\sum_{i,j}|t_i||s_j||\varphi(\log|t_i|^q)- \varphi(\log|s_j|^p)|\tau(y_iw_j)\\
&\leq & L_{\varphi}\sum_{i,j}|t_i||s_j|\left|\log\frac{|t_i|^q}{|s_j|^{p}}\right|\tau(y_iw_j)\\
&\leq &  L_{\varphi}\sum_{i,j}  \frac{p}{e}\left(|t_i|^q+ |s_j|^p\right)\tau(y_iw_j)\\
&=&\frac{p}{e}L_{\varphi}\left( \sum_i|t_i|^q\left( \sum_j\tau(y_iw_j)    \right)  +  \sum_j|s_j|^p \left( \sum_i\tau(y_iw_j)   \right)        \right)\\
&=&\frac{p}{e}L_{\varphi}\left( \sum_i|t_i|^q \tau(y_i)    +  \sum_j|s_j|^p \tau(w_j)         \right)\\
&=& \frac{2pL_{\varphi}}{e},
\end{eqnarray*}
where $L_{\varphi}$ denotes the Lipschitz constant of $\varphi$. Assuming for instance that $1<p\leq 2$, by homogeneity
 \begin{equation}\label{suf2}
|\tau(\Phi_q(y)w-y\Phi_p(w))|\leq 2 L_{\varphi}\|y\|_q\|w\|_p,
\end{equation}
whenever $y$ and $w$ are $\sigma$-elementary operators. Now, suppose $y$ and $w$ are self-adjoint. It is easy to find a sequences of
$\sigma$-elementary operators $(y_n)$ and $(w_n)$ such that the numerical sequences
$$
\|y_n-y\|_q,\quad \|\Phi_qy_n-\Phi_qy\|_q, \quad \|w_n-w\|_p,\quad \|\Phi_pw_n-\Phi_pw\|_p
$$
are all convergent to zero. This implies that
$$
\|(\Phi_q(y)w-y\Phi_p(w))-(\Phi_q(y_n)w_n-y_n\Phi_p(w_n))\|_1\to 0
$$
and so (\ref{suf2}) holds when  $y$ and $w$ are self-adjoint. Next, if $y\in L^q(\M,\tau)$ and $w\in\lt$ is self-adjoint, we can write $y=y_1+iy_2$, with each $y_i$ self-adjoint and since $\Phi_q$ is quasi-linear one has $\|\Phi_q(y)-\Phi_q(y_1)-i\Phi_q(y_2)\|_q\leq Q[\Phi_q](\|y_1\|_q+\|y_2\|_q)\leq 2Q[\Phi_q]\|y\|_q$ and
\begin{align*}
|\tau(\Phi_q(y)w-y\Phi_pw)|&= |\tau\left((\Phi_qy- \Phi_qy_1-i\Phi_qy_2 )w+(\Phi_qy_1+i\Phi_qy_2)w-(y_1+iy_2)\Phi_p(w)\right)|\\
&\leq   2Q[\Phi_q]\|y\|_q\|w\|_p+ 2 L_{\varphi}(\|y_1\|_q+\|y_2\|_q)\|w\|_p\\
&\leq  \left(2Q[\Phi_q]+ 4 L_{\varphi}\right)\|y\|_q\|w\|_p.
\end{align*}
Finally, writing $w=w_1+iw_2$ with each $w_i$ self-adjoint and using the quasilinearity of $\Phi_p$ one arrives to (\ref{suf}), where $M$ depends on $p, q$ and $L_\vp$, but not on $f$ or $g$.

Therefore, going back to (\ref{pairing}) we have
\begin{eqnarray*}
|\tau(xw-yv)|& = &|\tau((x-\Phi_q(y))w+\Phi_q(y)w-y(v-\Phi_p(w))-y\Phi_p(w))|\\
&\leq& \|(v-\Phi_p(w)\|_{p}\|y\|_{q} +
M\|w\|_{p}\|y\|_{q} +
\|w\|_{p}\|(x-\Phi_q(y)\|_{q}\\
&\leq& (M+1) \|(x,y)\|_{\Phi_q}\|(v,w)\|_{\Phi_p}.
\end{eqnarray*}
The remainder of the proof is quite easy: we have just seen that the map $u:Z_q^\vp(\tau)\to (Z_p^\vp(\tau))^*$ given by $(u(x,y))(v,w)=\tau(xw-yv)$ is bounded. On the other hand, the following diagram is commutative:
$$
\begin{CD}
L^q(\tau)@>>> Z_q^\vp(\tau) @>>>L^q(\tau)\\
@| @VV u V @VV -1 V\\
L^p(\tau)^*@>>> (Z_p^\vp(\tau))^* @>>>L^p(\tau)^*
\end{CD}
$$
Here, the lower row is the adjoint (in the Banach space sense) of the extension induced by $\Phi_p$. It follows that $u$ is one-to-one, onto, and open.
\end{proof}

\subsection{The role of the trace} Theorem~\ref{k-semi} cannot be extended to arbitrary centralizers. Actually, the following example shows that the behavior of $\Phi_\tau$ may depend strongly on the trace $\tau$.

\begin{example} For each $\pm$ and $p\in(1,\infty)$, consider the centralizer on $L^p(\R^+)$ given by
$
\Phi^{\pm}(f)= f(\imath^\pm(\log(|f|/\|f\|_p)))
$, where
$\imath^+(t)=\max\{0,t\}$ and $\imath^-(t)=\min\{0,t\}$
Then, with the notation of Theorem~\ref{main}:
\begin{itemize}
\item[(a)] $\Phi^{\pm}$ is nontrivial on $L^p(\R^+)$.
\item[(b)] If $\tau$ is bounded away from zero on the projections of $\M$ then $\Phi^+_\tau$ is trivial on $L^p(\M,\tau)$,  while $\Phi^-_\tau$ is nontrivial as long as $\M$ is infinite-dimensional.
\item[(c)] If $\tau(\bf 1_\M)<\infty$  then $\Phi^-_\tau$ is trivial on $L^p(\M,\tau)$, while $\Phi^+_\tau$ is nontrivial as long as $\M$ is infinite-dimensional.
\end{itemize}
\end{example}
\begin{proof} (See the proof of Theorem~\ref{k-semi}.)
Let $\Psi:L^p\to L^0$ be any centralizer. Let $(A_i)$ be a sequence of disjoint measurable sets, with finite and positive measure and let $\mathscr A$ be the least $\sigma$-algebra of Borel sets containing every $A_i$. Then, if $\Psi$ maps $L^p(\R^+,\mathscr A,\lambda)$ to $L^0(\R^+,\mathscr A,\lambda)$, in particular if $\Psi$ is lazy, then it defines an $L^\infty(\R^+,\mathscr A,\lambda)$ centralizer on $L^p(\R^+,\mathscr A,\lambda)$. Moreover, if $\Psi$ is trivial on $L^p(\R^+,\mathscr A,\lambda)$ (as a quasi-linear map), then it is also trivial as an $L^\infty(\R^+,\mathscr A,\lambda)$ centralizer.

(a) To check that $\Phi^+$ is nontrivial on $L^p(\R^+)$ just take a sequence $(A_i)$ with $|A_i|=2^{-i}$. To check that $\Phi^-$ is nontrivial, take $A_i$ with $|A_i|= 1$ for all $i\in\N$.

(b) We may assume $\tau(e)\geq 1$ for every projection $e\in\M$. Pick a positive, $\sigma$-elementary $f$ normalized in $L^p(\tau)$ so that $f=\sum_{n=1}^\infty f_ie_i$, with $f_i\geq 0$ and $e_i$ disjoint projections. Obviously $f_i\leq 1$ for every $i$ and so $\Phi^+(f)=0$. It follows that $\Phi^+$ is bounded on $L^p(\tau)$.

As $\Phi^+_\tau +\Phi^-_\tau=\Omega_p$ and $\Phi^+$ is trivial we see that $\Phi^-$ must be nontrivial since $\Omega_p$ is nontrivial unless $\M$ is finite-dimensional.

(c) We may assume $\tau({\bf 1}_\M)= 1$. Take a positive, normalized $f\in L^p(\M,\tau)$ and write $f=\int_0^\infty\lambda de(\lambda)$ to be its spectral resolution. Set $g=\int_0^1\lambda de(\lambda)$ and $h=\int_{1^+}^\infty\lambda de(\lambda)$. One has
$$
\| \Phi^-_\tau(f)-\Phi^-_\tau(g)-\Phi^-_\tau(h)  \|_p\leq Q[\Phi^-_\tau](\|g\|_p+\|h\|_p)\leq 2Q[\Phi^-_\tau].
$$
Obviously, $\Phi^-_\tau(h)=0$, while $g\in\M$, with $\|g\|_\infty\leq 1$. Hence
$$
\|\Phi^-_\tau(g)-g\Phi^-_\tau({\bf 1}_\M)  \|_p\leq C[\Phi^-_\tau]\|g\|_\infty\|1\|_p\leq C[\Phi^-_\tau].
$$
But $\Phi^-_\tau({\bf 1}_\M)=0$ and so  $\|\Phi^-_\tau(f)\|_p\leq 2Q[\Phi^-_\tau]+ C[\Phi^-_\tau]$.
\end{proof}

\subsection{Commutativity and symmetry}
The centralizers $\Phi_\tau$ appearing in  Theorem~\ref{main} have the property that, if $x\in\lt$ is self-adjoint, then $x$ and $\Phi_\tau(x)$ commute. This is not by accident. Indeed, suppose that $\Psi:\lt\to\widetilde\M$ is any bicentralizer and that $x$ is selfadjoint. Let $\mathcal A$ be a maximal abelian self-adjoint subalgebra containing the spectral projections of $x$, so that $ax=xa$ for every $a\in\mathcal A$. Then
$$
\|a\Psi x- (\Psi x)a\|_p= \|a\Psi x- \Psi(ax)+\Psi(xa)-(\Psi x)a\|_p\leq 2C[\Psi]\|a\|\|x\|,
$$
and $\|\Psi(x)-u(\Psi x)u^*\|\leq M\|x\|$ for every unitary $u\in \mathcal A$. Averaging the difference $\Psi(x)-u(\Psi x)u^*$ over the unitary group of $\mathcal A$ one obtains an element $B(x)\in\lt$ such that $\|B(x)\|_p\leq M\|x\|_p$ and such that $\Psi(x)-B(x)$ commutes with $\mathcal A$. Thus, if we define $\tilde\Psi(x)=\Psi(x)-B(x)$ we get a centralizer with the additional property that $x$ is self-adjoint, then $\tilde\Psi(x)$ commutes with (the spectral projections of) $x$.

\medskip

One may wonder what is the role of the symmetry of the starting centralizer $\Phi$ in Theorem~\ref{main}. In general one cannot expect to get bicentralizers out from arbitrary (lazy) centralizers, as shown by Kalton in \cite[Theorem 8.3]{k-tams}. And this is so because, if $\M$ is large enough, the bimodule structure of $\lt$ already encodes the ``symmetric'' structure of its ``commutative'' subspaces. Actually even the definition of $X(\M,\tau)$ requires the symmetry of the function space $X$.
To explain this, let us consider the following situation. Let $\mathcal H=\ell^2$ be the standard Hilbert space of 2-summable sequences $f:\N\to\C$ and consider the algebra $B(\mathcal H)$ of all bounded operators on $\mathcal H$, with the usual trace. Then the corresponding $L^p$ spaces are just the Schatten classes $S^p$.

Each bounded sequence $b\in\ell^\infty$ induces a multiplication operator $M_b(f)=b\cdot f$, which is ``diagonal'' with respect to the unit basis of $\mathcal H$.

It is clear from the preceding remark that if $\Psi$ is any bicentralizer on $S_p$, then one may assume that $\Psi(x)$ is ``diagonal'' whenever $x$ is so. Since diagonal operators in $S_p$ correspond with multiplication operators by a sequence in $\ell^p$ we see that $\Psi$ gives rise to a mapping $\psi$ (actually an $\ell^\infty$-centralizer) on $\ell^p$ defined by $\Psi(M_f)=M_{\psi(f)}$.

Let us see that $\psi$ must be symmetric. Indeed, let $u$ be a permutation of $\mathbb N$ and consider the isometry of $\mathcal H$ given by $U(h)=h\circ u$. Then $U^*=U^{-1}$ is given by $h\mapsto h\circ u^{-1}$.
Note that if $b\in\ell^\infty$, then $UM_bU^*=M_{b\circ u}$ since for $h\in\mathcal H$
$$
UM_bU^*(h)=U(M_b(h\circ u^{-1}))=U(b\cdot(h\circ u^{-1}))=(b\circ u)\cdot h=M_{b\circ u}(h).
$$
Thus, if $f\in\ell^p$, and taking $b=\psi(f)$, we have
\begin{align*}
\|\psi(f\circ u)&-(\psi(f))\circ u\|_{\ell^p}= \|M_{\psi(f\circ u)}-M{(\psi(f))\circ u}\|_{S^p}=
\|\Psi(M_{f\circ u})-U\Phi(M_f)U^*\|_{S_p}\\
&= \|U\Psi(M_{f})U^*-U\Phi(M_f)U^*\|_{S_p}\leq C[\Psi]\|M_f\|_{S^p}=C[\Psi]\|f\|_{\ell^p}
\end{align*}
and $\psi$ is symmetric.

\section{Type III algebras}
In this Section we abandon the comfortable tracial setting and we face the problem of twisting arbitrary $L^p$ spaces, including those built over type III von Neumann algebras. There are several constructions of these $L^p$ spaces, none of them elementary. All provide bimodule structures on the resulting spaces that turn out to be equivalent at the end.

It is natural to ask for (nontrivial) self-extensions of $L^p(\M)$ in the category of Banach bimodules over $\M$. Unfortunately we have been unable to construct such objects; 
nevertheless we can still use the interpolation trick to obtain self extensions as (one-sided) modules. In this regard the most suited representation of $L^p$ spaces is one due to Kosaki.

For the sake of clarity, we can restrict here to $\sigma$-finite algebras so that we can take functionals from $\M_*$. So, let $\M$ be a von Neumann algebra and $\phi\in \M_*$ a faithful positive functional. (We don't normalize $\phi$ because the restriction of a state to a direct summand is not a state; see Lemma~\ref{arranged}(b) below.) We ``include'' $\M$ into $\M_*$ just taking $a\in \M\mapsto a\phi\in\M_*$ thus starting the interpolation procedure with $\Sigma=\M_*$ as ``ambient'' space and $\Delta=\M\phi$, to which the norm and $\sigma(\M,\M_*)$ topology are transferred without further mention. Then, the Kosaki (left) version of the space $L^p(\M)$ is defined as
$$L^p(\phi)=L^p(\M,\phi)=[\M\phi,\M_*]_\theta,\quad\quad(\theta=1/p).$$

We emphasize we are referring to Kosaki's construction \cite{kosaki, raynaud, pisier-handbook} and not to that of Terp \cite{terp1, terp2}.
Recall that $\M_*$ is an $\M$-bimodule with product given by
$$
\langle a\psi b, x\rangle= \langle \psi, bxa\rangle\quad\quad(\psi\in\M_*; a,b,x\in\M).
$$
The inclusion
$\cdot\phi: \M\to \M_*$ is, however, only a left-homomorphism:
$
(ba)\cdot\phi= b(a\phi).
$
Asking for a two-sided homomorphism means that one should also have
$$
(ab)\cdot\phi=ab\phi = a\phi b =(a\phi)\cdot b .
$$
In particular (take $a=1$) $b\phi=\phi b$ for all $b\in\M$, which happens if and only if $\phi$ is a trace.

Let $\mathcal G=\mathcal G(\M,\phi)$ denote the Calder\'on space associated to the couple $(\M\phi,\M_*)$ and put $\mathcal G_0=\mathcal G(\M,\phi)_0=\{g\in\cl{G}:g(\theta)=0\}$, where $\theta=1/p$ is fixed. These are left $\M$-modules in the obvious way and so are the quotients $L^p(\M,\phi)=\cl{G}/\cl{G}_0$.
Plug and play to get the push-out diagram
\begin{equation}\label{kosaki}
\begin{CD}
\cl{G}_0=\ker\delta_\theta @>>> \mathcal G @>\delta_\theta >>L^p(\M,\phi)\\
@V\delta'_\theta VV @VVV @| \\
L^p(\M,\phi)@>>> \PO @>>>L^p(\M,\phi)
\end{CD}
\end{equation}
(where $\theta=1/p$) and observe that every arrow here is a homomorphism of left $\M$-modules. Let us denote by $Z_p(\M,\phi)$ or
$Z_p(\phi)$ the push-out space in the preceding diagram. This is coherent with the notation used in the tracial case.

We have mentioned that there is also a right action of $\M$ on $L^p(\M,\phi)$ which is compatible with the given left action and makes $L^p(\M,\phi)$ into a bimodule. All known descriptions of that action are quite heavy and depend on Tomita-Takesaki theory.
That action is in general incompatible 
with the arrows in the preceding diagram.

Now, we are confronted with the problem of deciding whether the lower extension in Diagram~\ref{kosaki} is trivial or not. The pattern followed in the proof of Theorem~\ref{k-semi} cannot be used now because we have only a left multiplication in $Z_p(\M,\phi)$.

Suppose we are given two von Neumann algebras $\M$ and $\mathcal N$ with distinguished faithful normal states $\phi$ and $\psi$. If $u_\infty:\M\to \mathcal N$ and $u_1:\M_*\to \mathcal N_*$ are operators making the square
$$
\begin{CD}
\M@> u_\infty >> \mathcal N\\
@V \cdot\phi  VV  @VV  \cdot\psi  V\\
\M_*@>  u_1 >> \mathcal N_*
\end{CD}
$$
commutative,
then interpolation yields operators
$u_p : L^p(\M,\phi)\to L^p(\mathcal N,\psi)$ for each $p\in(1,\infty)$.

\begin{lemma}\label{arranged}
Let $\M$ be a von Neumann algebra with a faithful positive normal functional $\phi$. Let $\mathcal N$ be a subalgebra of $\M$ equipped with the restriction of $\phi$. Suppose either
\begin{itemize}
\item[(a)] $\mathcal N$ is a von Neumann subalgebra of $\M$ and there is a normal conditional expectation $\e:\M\to \mathcal N$ leaving $\phi$ invariant; or
\item[(b)] $\mathcal N$ is a von Neumann algebra, and a direct summand in $\M$.
\end{itemize}
Then, for each $p\in[1,\infty]$, there are homomorphisms of $\mathcal N$-modules $\imath_p : L^p(\mathcal N,\phi|_\mathcal N)\to L^p(\M,\phi)$ and $\e_p: \lt\to  L^p(\mathcal N,\phi|_\mathcal N)$ such that $\e_p\circ \imath_p$ is the identity on $ L^p(\mathcal N,\phi|_\mathcal N)$.
\end{lemma}

\begin{proof} (a)
We have assembled the hypotheses in order to guarantee the commutativity
of the diagram
\begin{equation}\label{assembled}
\begin{CD}
\mathcal N @> \imath >> \M@> \e >> \mathcal N\\
@V \cdot\phi|_\mathcal N  VV @V \cdot\phi  VV   @VV  \cdot\phi|_\mathcal N  V\\
\mathcal N_*@>\e_* >> \M_*@>  \imath_* >> \mathcal N_*
\end{CD}
\end{equation}
Here, $\imath:\mathcal N\to \M$ the inclusion map and the subscript indicates preadjoint (in the Banach space sense), in particular $\imath_*$ is plain restriction.

Indeed, for $a\in \mathcal N$, one has $\e_*(a\phi)=a\e_*(\phi)=a\phi$, so the left square commutes. As for the right one, taking $a\in \mathcal N, b\in \mathcal M$ we have
$$
\langle \e(b)\phi, a\rangle =
\langle \phi, a\e(b)\rangle =
\langle \phi, \e(ab)\rangle =
\langle \phi, ab\rangle =
\langle b\phi, a\rangle.
$$
Notice, moreover, that $\e\circ\imath$ is the identity on $\mathcal N$, while $\imath_*\circ\e_*$ is the identity on $\mathcal N_*$. Therefore, interpolating $(\imath, \e_*)$ we get operators $\imath_p:L^p(\mathcal N,\phi|_\cl{N}\to L^p(\mathcal M,\phi)$ for $1\leq p\leq \infty$, while $(\e,\imath_*)$ gives operators $\e_p:L^p(\mathcal M,\phi)\to L^p(\mathcal N,\phi|_\cl{N})$. And since $\e_p\circ \imath_p$ is the identity on $L^p(\mathcal N,\phi|_\cl{N})$ we are done.

(b) In this case we can use the same diagram, just replacing $\varepsilon$ by the projection $P:\M\to\mathcal N$ given by $P(a)=e ae$, where $e$ is
the unit of $\mathcal N$. Then $P_*:\mathcal N_*\to\M_*$ is given by $\langle P_*(\psi), b\rangle=  \langle \psi, ebe\rangle$.
\end{proof}

The following step is the result we are looking for.

\begin{lemma}\label{complemented}
With the same hypotheses as in Lemma~\ref{arranged}, $Z_p(\mathcal N,\phi|_\mathcal N)$ is a complemented subspace of $Z_p(\M,\phi)$ for every $1< p<\infty$.
\end{lemma}


\begin{proof}We write the proof assuming (a). The other case requires only minor modifications that are left to the reader.
Let us begin with the embedding of $\PO(\mathcal N)=Z_p(\mathcal N,\phi|_{\mathcal N})$ into $\PO(\mathcal M)= Z_p(\mathcal M,\phi)$. Consider the diagram
$$
\xymatrixcolsep{1pc}
\xymatrix{
\mathcal{G}(\mathcal{N},\phi|_{\mathcal{N}})_0 \ar[rr]\ar[rd]\ar[dd]_{\delta_\theta'} & & \mathcal{G}(\mathcal{N},\phi|_\cl{N})  \ar[rr]^{\delta_\theta} \ar[rd]_{(\varepsilon_{*})_\circ} \ar[dd] &  & L^p(\mathcal{N}) \ar[rd] \ar@{=}[dd] \\
& \mathcal{G}(\mathcal{M},\phi)_0 \ar[rr] \ar[dd]_<<<<<<<{\delta_\theta'}  & & \mathcal{G}(\mathcal{M},\phi ) \ar[dd] \ar[rr]^{\delta_\theta\quad\quad}&  & L^p(\mathcal{M}) \ar@{=}[dd]\\
L^p(\mathcal N) \ar[rr]\ar[dr] & & \PO(\mathcal N) \ar[rr] \ar@{.>}[dr] & & L^p(\mathcal N) \ar[dr]\\
& L^p(\mathcal M) \ar[rr] & & \PO(\mathcal M) \ar[rr] & & L^p(\mathcal M)
}
$$

Here, $(\varepsilon_{*})_\circ$ sends a given function $f:\mathbb S\to \mathcal N_*$ to the composition $\varepsilon_*\circ f:\mathbb S\to\mathcal N_*\to\M_*$ and the
 mappings from $L^p(\mathcal N)$ to $L^p(\M)$ are all given by $\imath_p$.
 It is not hard to check that this is a commutative diagram. Therefore, we can insert an operator
$\kappa: \PO(\mathcal N)\to \PO(\mathcal M)$ making the resulting diagram commutative because of the universal property of the push-out square
$$
\begin{CD}
\mathcal G(\mathcal N,\phi|_\cl{N})_0 @>>> \mathcal G(\mathcal N,\phi|_\cl{N})\\
@V\delta'_\theta VV @VVV \\
L^p(\mathcal N)@>>> \PO(\mathcal N)
\end{CD}
$$
A similar argument shows the existence of an operator $\pi: \PO(\mathcal M)\to \PO(\mathcal N)$ making commutative the diagram
$$
\xymatrixcolsep{1pc}
\xymatrix{
\mathcal{G}(\mathcal{M},\phi)_0 \ar[rr]\ar[rd]\ar[dd]_{\delta_\theta'} & & \mathcal{G}(\mathcal{M},\phi)  \ar[rr]^{\delta_\theta} \ar[rd]_{(\imath_{*})_\circ} \ar[dd] &  & L^p(\mathcal{M}) \ar[rd] \ar@{=}[dd] \\
& \mathcal{G}(\mathcal{N},\phi|_{\mathcal{N}})_0 \ar[rr] \ar[dd]_<<<<<<<{\delta_\theta'}  & & \mathcal{G}(\mathcal{N},\phi|_{\mathcal N} ) \ar[dd] \ar[rr]^{\delta_\theta\quad\quad}&  & L^p(\mathcal{N}) \ar@{=}[dd]\\
L^p(\mathcal M) \ar[rr]\ar[dr] & & \PO(\mathcal M) \ar[rr] \ar[dr]_\pi & & L^p(\mathcal M) \ar[dr]\\
& L^p(\mathcal N) \ar[rr] & & \PO(\mathcal N) \ar[rr] & & L^p(\mathcal N)
}
$$
The arrows from $L^p(\M)$ to $L^p(\mathcal N)$ are now given by $\varepsilon_p$. Putting together the two preceding
diagrams it is easily seen that $\pi\circ\kappa$ is the identity on $\PO(\mathcal N)$.
\end{proof}

Here is the main result about the twisting of Kosaki's $L^p$. As we shall see later (Section~\ref{change}) $Z_p(\M,\phi)$ doesn't depend on $\phi$ and so the conclusion of the following Theorem holds for any $\phi$.

\begin{theorem}
Let $\M$ be an infinite dimensional von Neumann algebra. There is a faithful normal state $\phi$ for which the lower extension of the push-out diagram
$$
\begin{CD}
\ker\delta_\theta @>>> \mathcal G(\M,\phi) @>\delta_\theta >>L^p(\M,\phi)\\
@V\delta'_\theta VV @VVV @| \\
L^p(\M,\phi)@>>> \PO @>>>L^p(\M,\phi)
\end{CD}
$$
is nontrivial.
\end{theorem}

\begin{proof}
The idea of the proof is to choose $\phi$ in such a way that its ``centralizer subalgebra''
$$
\M^\phi=\{a\in \M: a\phi=\phi a\}
$$
is infinite dimensional.
After that we proceed as follows. By \cite[Corollary III.4.7.9]{black}, there is a normal conditional expectation $\e:\M\to \M^\phi$ leaving $\phi$ invariant: $\phi=\phi|_{\M^\phi}\circ \e$. Actually $\e$ is unique, by \cite[Corollary II.6.10.8]{black}.

Apply now Lemma~\ref{complemented} to embed $\PO(\M^\phi, \phi)$ as a complemented subspace (in fact as a ``complemented subextension'') of $\PO(\M, \phi)$ and please note that the restriction of $\phi$ to $\M^\phi$ is a (finite) trace by the very definition of $\M^\phi$.

The nonsplitting of $\PO(\M^\phi, \phi)$ is nothing but a particular case of Theorem~\ref{k-semi} as for a finite trace $\tau$ one has $\M\subset L^1(\tau)$ and, after identifying $L^1(\tau)$ with $\M_*$, the inclusion agrees with Kosaki's left method.

In order to find out the required $\phi$, let us decompose $\M=\mathcal N\oplus\mathcal L$, with $\mathcal N$ semifinite and $\mathcal L$ without direct summands of type I (This can be done in several ways: for instance, taking $\mathcal N$ as the semifinite part and $\mathcal L$ as the type III part of $\M$, or taking $\mathcal N$ as the discrete part and $\mathcal L$ as the continuous part, see \cite[Section~III.1.4]{black}.)

By Lemma~\ref{complemented} we have an isomorphism
$\PO(\M,\phi)=\PO(\mathcal N,\phi|_\mathcal N)\oplus \PO(\mathcal L,\phi|_\mathcal L)$ and we can consider the two cases separately.

$\bigstar$
First, assume $\M$ has no direct summand of type I (so that it is either type II or III). Then, if $\psi$ is any faithful normal state on $\M$, there is  a faithful normal state $\phi$ (in the closure of the orbit of $\psi$ under the inner automorphisms of $\M$) whose centralizer
subalgebra $\M^\phi$
is of type II$_1$ (\cite[Theorem 11.1]{haa-sto}) and we are done.

$\bigstar$  Now, suppose $\M$ semifinite and let us see that any $\phi$ works. Let $\tau$ be a (fns) trace on $\M$ and let us identify $\M_*$ with $L^1(\tau)$ so that we may consider $\phi$ as a $\tau$-measurable operator on the ground Hilbert space. If $\phi$ is elementary, let us write it as $\phi=\sum_{i=1}^nt_ie_i$, where the $e_i$ are mutually orthogonal projections in $\M$. Letting $\M_i=e_i\M e_i$ we see that $\oplus_i\M_i$ is an infinite dimensional subalgebra of $\M^\phi$, which is enough.
Otherwise $\phi$ has infinite spectrum and its spectral projections already generate an infinite dimensional subalgebra of $\M^\phi$.
\end{proof}

\subsection{Duality again} We now give a description of the dual of $Z_p(\M,\phi)$ for general $\M$. To this end we consider the right embedding of $\M$ into $\M_*$ given by $a\mapsto \phi a$ which is a homomorphism of right modules and the new couple $(\phi\M,\M_*)$. The former couple using the left embedding is denoted by $(\M\phi,\M_*)$. The right version of Kosaki $L^p$ is
$$
L^p(\M,\phi)^r=[\phi\M,\M_*]_{1/p}= [\M_*,\phi\M]_{1-1/p}.
$$
Let us define $Z_p(\M,\phi)^r$ as the push-out space (actually right module on $\M$) in the ubiquitous diagram
\begin{equation}\label{kosakir}
\begin{CD}
\ker\delta_\theta @>>> \mathcal G(\M,\phi)^r @>\delta_\theta >>L^p(\M,\phi)^r\\
@V\delta'_\theta VV @VVV @| \\
L^p(\M,\phi)^r@>>> \PO @>>>L^p(\M,\phi)^r
\end{CD}
\end{equation}
where $\theta=1/p$ and  $\mathcal G(\M,\phi)^r $ is the Calder\'on space associated to the couple $(\phi\M,\M_*)$.

We want to see that if $p,q\in(1,\infty)$ are conjugate exponents, then the conjugate of $Z_p(\M,\phi)^\ell$ (our former  $Z_p(\M,\phi)$) is well isomorphic to  $Z_q(\M,\phi)^r$.

Consider the couples $(\M\phi,\M_*)$ and $(\M_*,\phi\M)$ (not $(\phi\M,\M_*)$!). Then
$$\Delta^\ell= \Delta(\M\phi,\M_*)=\M\phi \quad\quad\text{and}\quad\quad \Delta^r= \Delta(\M_*,\phi\M)=\phi\M.
$$
Both $\M\phi$ and $\phi\M$ are dense in $\M_*$ since $\phi$ is faithful.
Define a bilinear form $\beta:\Delta^\ell\times \Delta^r\to\C$ by $\beta(a\phi, \phi b)= \phi(ba)$.
The key point in that
$$
\beta(a\phi, \phi b)=\langle a\phi, b\rangle = \langle a, \phi b\rangle,
$$
where the brackets refer to the dual pairing between $\M_*$ and $\M$. (Notice, moreover, that $\beta$ is balanced in the sense that $\beta(cf,g)=\beta(f,gc)$ for $f\in\Delta^\ell, g\in\Delta^r$ and $c\in\M$.)

Then $\beta$ is bounded both at $\theta=0$ and $\theta= 1$. Indeed, for $\theta=0$ one has
$$
|\phi(ba)|\leq \|a\|_\M\|\phi b\|_{\M_*}.
$$
Similarly, when $\theta=1$,
$$
|\phi(ba)|\leq \|a\phi\|_{\M_*}\|b\|_{\M}.
$$
By bilinear interpolation \cite[Theorem 4.4.1]{bergh} $\beta$ extends to a bounded bilinear form on $L^p(\M,\phi)^\ell\times  L^q(\M,\phi)^r=[\M\phi,\M_*]^\ell_\theta\times [\M_*,\phi\M]^r_\theta$ which provides the dual pairing between $L^p(\M,\phi)^\ell$ and $ L^q(\M,\phi)^r$ (see \cite{kosaki} or \cite{raynaud}). Let us call $\beta$ to that extension.

For $1<p<\infty$, let $\Omega_p^\ell:L^p(\M,\phi)\to \M_*$ be the derivation associated to the identity $[\M\phi,\M_*]^\ell_{1/p}=L^p(\M,\phi)$ and
$\Omega_p^r:L^p(\M,\phi)^r\to \M_*$ that associated to $[\phi\M,\M_*]^r_{1/p}=L^p(\M,\phi)^r$. Note that if $\theta=1/p$, then the derivation associated to $[\M_*,\phi\M]^r_{\theta}=L^q(\M,\phi)^r$ is just $-\Omega_q^r$. Proposition 1.3 in \cite{fan} yields
$$
|\beta(\Omega_p^\ell(f),g)-\beta(f,\Omega_q^r(g))|\leq \frac{\pi}{\sin (\pi\theta)}\|f\|_{p}\|g\|_q
$$
at least when $f\in \M\phi$ and $g\in\phi\M$. The following result is implicit in \cite{rochberg-weiss}.

\begin{theorem} Given conjugate exponents $p,q\in(1,\infty)$, the dual of $$Z_p(\M,\phi)^\ell= L^p(\M,\phi)^\ell\oplus_{\Omega_p^\ell}L^p(\M,\phi)^\ell$$ is isomorphic to $$-Z_q(\M,\phi)^r= L^q(\M,\phi)^r\oplus_{-\Omega_q^r}L^q(\M,\phi)^r.$$ More precisely, there is an isomorphism of right Banach modules over $\M$ making commutative the following diagram
\begin{equation}\label{dual}
\begin{CD}
(L^p(\M,\phi)^\ell)^*@>\pi^*>> (Z_p(\M,\phi)^\ell)^* @>\imath^*>>(L^p(\M,\phi)^\ell)^*\\
@| @AA u A @| \\
L^q(\M,\phi)^r@>>> L^q(\M,\phi)^r\oplus_{-\Omega_q^r}L^q(\M,\phi)^r @>>>L^q(\M,\phi)^r
\end{CD}
\end{equation}
\end{theorem}

\begin{proof}
Put
\begin{equation}\label{u}
(u(g',g))(f',f)=\beta(f,g')+\beta(f',g)\quad\quad(g\in \Delta^r, f\in \Delta^\ell).
\end{equation}
We have
$$
\begin{aligned}
|\beta(f,g')+\beta(f',g)| &= |\beta(f,g'+\Omega_q^r(g))- \beta(f,\Omega_q^r(g))+\beta(f'-\Omega_p^\ell(f),g)+ \beta(\Omega_p^\ell(f),g)| \\
&\leq \| f\|_p\|g'+\Omega_q^r(g)\|_q+  \|f'-\Omega_p^\ell(f)\|_p\|g\|_q+\frac{\pi}{\sin (\pi/p)}\|f\|_{p}\|g\|_q\\
&\leq \frac{\pi}{\sin (\pi/p)}\|(g',g)\|_{-\Omega_q^r}\|(f',f)\|_{\Omega_p^\ell}.
\end{aligned}
$$
As $\Delta^\ell$ is dense in $L^p(\M,\phi)^\ell$, we see that $L^p(\M,\phi)^\ell\oplus_{\Omega_p^\ell}\Delta^\ell$ is dense in $Z_p(\M,\phi)^\ell$ and so (\ref{u}) shows that $u(g',g)$ acts, as a bounded linear functional on $Z_p(\M,\phi)^\ell$, with
$$\|u(g',g): Z_p(\M,\phi)^\ell\to\mathbb C\|\leq M\|(g',g)\|_{-\Omega_q^r},$$
at least when $g$ is in $\Delta^r$. This defines an operator making the following diagram commute:
\begin{equation}\label{dual2}
\begin{CD}
(L^p(\M,\phi)^\ell)^*@>\pi^*>> (Z_p(\M,\phi)^\ell)^* @>\imath^*>>(L^p(\M,\phi)^\ell)^*\\
@| @AA u A @| \\
L^q(\M,\phi)^r@>>> L^q(\M,\phi)^r\oplus_{-\Omega_q^r}\Delta^r @>>>\Delta^r
\end{CD}
\end{equation}
and where $\Delta^r$ is treated as a submodule of $L^q(\M,\phi)^r$. By density $u$ extends to an operator that we still call $u$ fitting in (\ref{dual}). The five-lemma and the open mapping theorem guarantee that $u$ is a linear homeomorphism. It remains to check it is also a homomorphism of right $\M$-modules. But for $g\in\Delta^r$ and $f\in\Delta^\ell$ one has
$$
\begin{aligned}
u((g',g)a)(f',f)&= (u(g'a,ga))(f',f)=\beta(f,g'a)+\beta(f',ga)\\
&=\beta(af,g')+\beta(af',g)=u(g',g)(af',af)=(u(g',g)a)(f',f).
\end{aligned}
$$
This completes the proof.
\end{proof}

\subsection{Change of state}\label{change}
In this Section we prove the extension $L^p(\M,\phi)\to Z_p(\M,\phi)\to L^p(\M,\phi)$ is  essentially independent on the reference state $\phi$ in the following precise sense.

\begin{proposition}
Let $\phi_0$ and $\phi_1$ be faithful normal states on $\M$ and $p\in(1,\infty)$. Then there is a commutative diagram
$$
\begin{CD}
L^p(\M,\phi_0)@>>> Z_p(\M,\phi_0)@>>> L^p(\M,\phi_0)\\
 @V\alpha VV @VVV @VV\alpha V \\
L^p(\M,\phi_1)@>>> Z_p(\M,\phi_1)@>>> L^p(\M,\phi_1)
\end{CD}
$$
in which the vertical arrows are isomorphisms of left $\M$-modules.
\end{proposition}

\begin{proof} The proof is based on an idea explained and discarded by Kosaki in \cite[p.~71]{kosaki}. We remark that our proof provides a very natural isometry between $L^p$ spaces based on two different states.

It will be convenient to consider two more spaces of analytic functions. The first one is the obvious adaptation of the space $\mathcal H$ appearing in Section \ref{kalton-peck} to the nontracial setting. So, given a faithful state $\phi\in\M_*$, we consider the couple $(\M\phi,\M_*)$, and the space $\mathcal H=\mathcal H(\M,\phi)$ of bounded functions $H:\S\to\M_*$ such that:
\begin{itemize}
\item[(1)] $H$ is continuous on $\S$ and analytic on $\S^\circ$ with respect to $\sigma(\M_*,\M)$.
\item[(2)] $H(it)\in\M\phi$ for every $t\in\R$. The function $t\in\R\mapsto H(it)\in\M\phi$ is $\M$-bounded and $\sigma(\M,\M_*)$-continuous.
\item[(3)] The function $t\in\R\mapsto H(1+it)\in\M_*$ is continuous in the norm of $\M_*$.
\end{itemize}
As one may expect we furnish $\mathcal H$ with the norm $\|H\|_{\mathcal H}=\sup_t\left(\|H(it)\|_\M,\|H(1+it)\|_{\M_*}\right)$.
Of course, $\mathcal H$ is larger than $\mathcal G$. The second space we shall denote by $\mathcal F=\mathcal F(\M,\phi)$ is the space of those $f\in\mathcal G(\M,\phi)$ satisfying the additional condition that $f(it)\to 0$ in $\M=\M\phi$ as $|t|\to\infty$ and  $f(1+it)\to 0$ in $\M_*$ as $|t|\to\infty$.
Moreover the complex method of interpolation, applied to the couple $(\M,\M_*)$, leads to the same scale using $\mathcal F,\mathcal G$ or $\mathcal H$:
$$
[\M,\M_*]^{\mathcal F}_\theta =  [\M,\M_*]^{\mathcal G}_\theta = [\M,\M_*]^{\mathcal H}_\theta =L^p(\M,\phi)\quad\quad(\theta=1/p)
$$
with identical norms. This is very easy to check, once we know that $L^p(\M,\phi)$ is reflexive and agrees with the dual of the right space $L^q(\M,\phi)^r$, where $q$ is the conjugate exponent of $p$. As Lemma~\ref{mechanism} is true (with the same proof) replacing $\mathcal G$ by  $\mathcal F$ or by $\mathcal H$ we see that the lower extension in Diagram~\ref{kosaki} does not vary after replacing $\mathcal G$ by  $\mathcal F$ or by $\mathcal H$.

We shall use the following notations:
\begin{align*}
\mathcal F_0(\M,\phi)&=\{F\in \mathcal F(\M,\phi): F(\theta)=0\},\\
\mathcal F_1(\M,\phi)&=\{F\in \mathcal F(\M,\phi): F(\theta)=F'(\theta)=0\}
\end{align*}
and similarly for $\mathcal G$ and $\mathcal H$. As we mentioned after Lemma~\ref{justdefined} one has isomorphisms
$$
Z_p(\M,\phi)=\frac{ \mathcal F(\M,\phi)}{  \mathcal F_1(\M,\phi)}= \frac{ \mathcal G(\M,\phi)}{  \mathcal G_1(\M,\phi)}= \frac{ \mathcal H(\M,\phi)}{  \mathcal H_1(\M,\phi)}.
$$
It is important to realize how these quotient spaces arise as self-extensions of $L^p=L^p(\M,\phi)$. We describe the details for the smaller space $\mathcal F$; replacing it by $\mathcal G$ or $\mathcal H$ makes no difference.

Recall that we have $\mathcal F_1\subset\mathcal F_0\subset \mathcal F$ and therefore an exact sequence
$$
\begin{CD}
0@>>> \mathcal F_0/\mathcal F_1@>\jmath>> \mathcal F/\mathcal F_1@>\varpi>> \mathcal F/\mathcal F_0@>>> 0
\end{CD}
$$
where $\jmath$ and $\varpi$ are the obvious maps. This becomes a self-extension of $L^p$ after identifying $\mathcal F/\mathcal F_0$ with $L^p$ through the (factorization) of the evaluation map $\delta_\theta:\mathcal F\to L^p$ at $\theta=1/p$, while the identification of  $ \mathcal F_0/\mathcal F_1$ with $L^p$ is provided by the (factorization) of the derivative $\delta_\theta':\mathcal F_0\to L^p$ (at $\theta=1/p$) which is an isomorphism of left modules over $\M$.

We conclude these prolegomena with the following observation. Let $\mathcal E(\M,\phi)$ denote the subspace of those $F\in  \mathcal F(\M,\phi)$ having the form $F(z)=f(z)\phi$, where $f:\mathbb S\to\M$ is continuous and analytic on the interior. It turns out that $\mathcal E(\M,\phi)$ is dense in  $\mathcal F(\M,\phi)$. Indeed, the set of functions having the form $F(z)=f(z)\phi$, with
$$
f(z)=\exp(\lambda z^2)\sum_{i=1}^n\exp(\lambda_i z)a_i\quad\quad(\lambda,\lambda_i\in\R, a_i\in\M)
$$
is already a dense subspace of $\mathcal F(\M,\phi)$. See \cite[Lemma 4.2.3]{bergh}.

Let $\vp:\S\to\mathbb D$ be the function given by (\ref{thatgiven}). Replacing $\mathcal G$ by $\mathcal F$ everywhere in the proof of Lemma~\ref{mechanism} we see that $\mathcal F_0=\vp\mathcal F$ in the sense that the multiplication operator $f\mapsto \vp f$ is an isomorphism between $\mathcal F$ and $\mathcal F_0$. Similarly,  $f\mapsto \vp^2 f$ is an isomorphism between $\mathcal F$ and $\mathcal F_1$. It follows that $\mathcal E\cap \mathcal F_0$ and $\mathcal E\cap \mathcal F_1$ are dense in $\mathcal F_0$ and in $\mathcal F_1$, respectively.

Now we need a bit of (relative) modular theory for which we refer the reader to \cite{kosaki} or \cite{raynaud}. We fix two faithful states $\phi_0,\phi_1\in\M_*$ and we consider the Connes-Radon-Nikod\'ym cocycle of $\phi_0$ relative to $\phi_1$:
$$
(D\phi_0;D\phi_1)_t=\Delta^{it}_{\phi_0\phi_1} \Delta^{-it}_{\phi_0}\quad\quad(t\in\R).
$$
As it happens, $t\mapsto (D\phi_0;D\phi_1)_t$ is a strongly continuous path of unitaries in $\M$ and so
\begin{equation}\label{path}
 t\longmapsto (D\phi_0;D\phi_1)_t\:\phi_1
\end{equation}
defines a continuous function from $\R$ to $\M_*$. Now the point is that (\ref{path}) extends to a function from the horizontal strip $-i\mathbb S=\{z\in\mathbb C: -1\leq \Im(z)\leq 0\}$ to $\M_*$ we may denote by $\overline{(D\phi_0;D\phi_1)_{(\cdot)}\phi_1}$ having the following properties:
\begin{itemize}
\item[(a)] For each $x\in\M$, the function $z\mapsto \langle \overline{(D\phi_0;D\phi_1)_{(z)}\phi_1}, x \rangle$ is continuous on $-i\mathbb S$ and analytic on $i\S^\circ$.
\item[(b)] $\overline{(D\phi_0;D\phi_1)_{(-i+t)}\phi_1}=\phi_0(D\phi_0;D\phi_1)_t$ for every real $t$.
\end{itemize}

We are going to define an isometric embedding of modules $I:\mathcal F(\M,\phi_0)\to \mathcal H(\M,\phi_1)$.
First, for $F\in\mathcal E(\M,\phi_0)$, we put
\begin{equation}\label{IF}
(IF)(z)=f(z) \overline{(D\phi_0;D\phi_1)_{(-iz)}\phi_1}\quad\quad(F(z)=f(z)\phi_0,z\in\S).
\end{equation}
We observe that for such an $F$ one has $\|F\|_\mathcal F=\max\{\|f(it)\|_\M, \|f(1+it)\phi_0\|_{\M_*}: t\in\R\}$.

Let us check that $IF\in\mathcal H(\M,\phi_1)$. That $IF$ satisfies (1) is obvious from (a).
Regarding the values of $IF$ on the boundary of $\mathbb S$ we have for real $t$:
\begin{equation}\label{it}
(IF)(it)=f(it)(D\phi_0;D\phi_1)_{t}\:\phi_1
\end{equation}
which certainly falls in $\M\phi_1$ since $(D\phi_0;D\phi_1)_{t}$ is unitary and, besides, $\|f(it)(D\phi_0;D\phi_1)_{t}\|_\M=\|f(it)\|_\M$.
Moreover, the function $t\in\R\mapsto f(it)(D\phi_0;D\phi_1)_{t}\in\M$ is $\sigma(\M,\M_*)$ continuous since  $t\in\R\mapsto f(it)\in\M$ is continuous for the norm and  $t\in\R\mapsto (D\phi_0;D\phi_1)_{t}\in\M$ is strongly (hence $\sigma(\M,\M_*)$) continuous. So (2) holds as well.

On the other hand,
$$
(IF)(1+it)=f(1+it)\overline{(D\phi_0;D\phi_1)_{(-i+t)}\phi_1}= f(1+it)\phi_0(D\phi_0;D\phi_1)_{t},
$$
so $\|(IF)(1+it)\|_{\M_*}= \| f(1+it)\phi_0(D\phi_0;D\phi_1)_{t}\|_{\M_*}=  \| f(1+it)\phi_0\|_{\M_*}$ and $
(IF)(1+it)$ is continuous in $t$ for the norm topology of $\M_*$. Finally, that $IF$ is $\M_*$ bounded on the whole $\S$ now follows by interpolation, using (\ref{it}). Hence $IF$ belongs to $\mathcal H(\M,\phi_1)$ and, moreover,
 the norm of $IF$ in $\mathcal H(\M,\phi_1)$ and the norm of $F$ in $\mathcal F(\M,\phi_0)$ coincide.

By density, $I$ extends to an isometric homomorphism of left $\M$-modules from $\mathcal F(\M,\phi_0)$ into $\mathcal H(\M,\phi_1)$ that we call again $I$.

Now we observe that $I$ maps $\mathcal F_0(\M,\phi_0)$ into $\mathcal H_0(\M,\phi_1)$. Indeed, it is obvious from (\ref{IF}) that
$IF$ vanishes at $\theta$ if $F\in\mathcal E$ vanishes at $\theta$ and for arbitrary $F$ the result follows by a density argument, taking into account that
$\mathcal H_0(\M,\phi_1)$ is closed in $\mathcal H(\M,\phi_1)$.
In particular,  $I$ induces a contractive homomorphism from $L^p(\M,\phi_0)$ to $L^p(\M,\phi_1)$.

Similarly, $I$ maps $\mathcal F_1(\M,\phi_0)$ into $\mathcal H_1(\M,\phi_1)$.
Indeed, for $F\in\mathcal E(\M,\phi_0)$ one has
$$
(IF)'(z)=f'(z) \overline{(D\phi_0;D\phi_1)_{(-iz)}\phi_1}+ f(z)\frac{d}{dz} \overline{(D\phi_0;D\phi_1)_{(-iz)}\phi_1}\quad\quad(F(z)=f(z)\phi_0).
$$
Thus, if $F(\theta)=F'(\theta)=0$ then $f(\theta)=f'(\theta)=0$ and therefore $(IF)(\theta)=(IF)'(\theta)=0$ and we proceed as before for general $F$.

Therefore we have a commutative diagram
$$
\xymatrixcolsep{1pc}
\xymatrix{
\mathcal{F}_0/\mathcal F_1 \ar[rr]\ar[rd]\ar[dd]_{\delta_\theta'} & & \mathcal{F}(\phi_0)/\mathcal F_1 \ar[rr] \ar[rd] \ar[dd]_<<<<<<<{(\delta_\theta',\delta_\theta)} &  & \mathcal{F}/\mathcal F_1 \ar[rd] \ar@{=}[dd]_<<<<<<<{\delta_\theta} \\
& \mathcal{H}_0/\mathcal H_1 \ar[rr] \ar[dd]_<<<<<<<{\delta_\theta'}  & & \mathcal{H}(\phi_1)/\mathcal H_1 \ar[dd]^<<<<<<<{(\delta_\theta',\delta_\theta)} \ar[rr]&  &\mathcal H/\mathcal H_0 \ar@{=}[dd]^{\delta_\theta} \\
L^p(\mathcal M,\phi_0) \ar[rr]\ar[dr]_\beta & & Z_p(\mathcal M,\phi_0) \ar[rr] \ar[dr]_\gamma & & L^p(\mathcal M,\phi_0) \ar[dr]_\alpha\\
& L^p(\mathcal M,\phi_1) \ar[rr] & & Z_p(\mathcal M,\phi_1)\ar[rr] & & L^p(\mathcal M,\phi_1)
}
$$
The arrows in the preceding Diagram can be described as follows. First, all arrows in the upper face going from left to right are the obvious ones.
All arrows in the upper face going from spaces based on $\phi_0$ to spaces based on $\phi_1$ are induced by $I$.

All vertical arrows are given by (factorization of) evaluations at $\theta=1/p$, as indicated in the diagram. They are all isomorphisms of left modules over $\M$. Thus, for instance $(\delta_\theta', \delta_\theta): \mathcal H(\phi_1)/\mathcal H_1\to Z_p(\mathcal M,\phi_1)$ takes (the class of) $H\in\mathcal H(\M,\phi_1)$ into the pair $(H'(\theta),H(\theta))\in  Z_p(\mathcal M,\phi_1)$ and so on.

The arrows lying in the bottom face are mere ``shadows'' of the corresponding arrows in the top face. It is really easy to see that  all arrows in the bottom face going from left to right act as expected. Let us identify the arrows of the bottom  face going from objects based on $\phi_0$ to objects based on $\phi_1$. We begin with $\alpha$. Suppose $x\in L^p(\phi_0)$ has the form $x=a\phi_0$, with $a\in\M$. Let $\e:\S\to\C$ be an analytic function such that $\e(\theta)= 1$ and $\e(\infty)=0$. Letting  $F(z)=\e(z)a\phi_0$ we have
$
(IF)(z)= \e(z)a \overline{(D\phi_0;D\phi_1)_{(-iz)} \phi_1}
$
 and so $\alpha(a\phi_0)=(IF)(\theta)=a \overline{(D\phi_0;D\phi_1)_{(-i\theta)} \phi_1}$.

In order to identify $\beta$ we take again $x=a\phi_0$ and we ``jump'' to $\mathcal F_0$ taking
$F(z)=(\vp'(\theta))^{-1}\e(z)\vp(z)a\phi_0$, where $\vp$ is the function defined by (\ref{thatgiven}). Note that
$$
\vp'(\theta)=\frac{\pi}{2}\frac{\exp(i\pi\theta)}{\sin(\pi\theta)}\neq 0.
$$
One has
$$
F'(\theta)=\frac{(\e\vp)'(\theta)}{\vp'(\theta)}a\phi_0=a\phi_0.
$$
Hence $\beta(a\phi_0)=(IF)'(\theta)$, where $(IF)(z)=(\vp'(\theta))^{-1}\e(z)\vp(z)a \overline{(D\phi_0;D\phi_1)_{(-iz)} \phi_1}$ and so by Leibniz's rule
$$
\beta(a\phi_0)=(IF)'(\theta)=a \overline{(D\phi_0;D\phi_1)_{(-i\theta)} \phi_1}=\alpha(a\phi_0).
$$
Finally, the same argument shows that
$$\gamma(b\phi_0,a\phi_0)= (b\overline{(D\phi_0;D\phi_1)_{(-i\theta)} \phi_1},a \overline{(D\phi_0;D\phi_1)_{(-i\theta)} \phi_1}).$$

To complete the proof we have to prove that $\alpha$ is an isomorphism -- that $\gamma$ is an isomorphism then follows from the five-lemma. This is not automatic because $I:\mathcal F(\phi_0)\to \mathcal H(\phi_1)$ is not surjective.

Anyway, reversing the r\^oles of $\phi_0$ and $\phi_1$ we know that there is a homomorphism of left $\M$ modules $\omega: L^p(\M,\phi_1)\to L^p(\M,\phi_0)$ such that
$$
\omega(x)=a(\overline{(D\phi_1;D\phi_0)_{(-i\theta)}\phi_0})\quad\quad(x=a\phi_1\in L^p(\phi_1)).
$$
We will prove that $\alpha$ and $\omega$ are inverse of each other.

To this end, let us say that $a\in\M$ is ``analytic'' if the map
$$
t\in\R\mapsto a (D\phi_0;D\phi_1)_t\in\M
$$
extends to an entire function we shall denote by $\overline{a (D\phi_0;D\phi_1)_{(\cdot)}}$. This is a ``left'' version of the usual definition; see, e.g., \cite[p. 73]{kosaki}. Let $\mathcal A$ denote the set of ``analytic'' operators in $\M$. It is not hard to see that the $\mathcal A$ is $\sigma$-weak ($=\sigma(\M,\M_*)$) dense in $\M$ and so the set
$
\mathcal A\phi_0=\{a\phi_0: a\in\mathcal A\}
$
is dense in $L^p(\phi_0)$. Thus the proof will be complete if we show that $\omega(\alpha(a\phi_0))=a\phi_0$ for $a\in\mathcal A$. But for such an $a$ we have
$$
\alpha(a\phi_0)=(\overline{a (D\phi_0;D\phi_1)_{(-i\theta})})\cdot\phi_1= a\cdot (\overline{(D\phi_0;D\phi_1)_{(-i\theta})\phi_1})
$$
by the uniqueness of analytic continuation. Therefore, as $\overline{a (D\phi_0;D\phi_1)_{(-i\theta)}}$ belongs to $\M$,
$$
\omega(\alpha(a\phi_0))=\omega\left( (\overline{a (D\phi_0;D\phi_1)_{(-i\theta)}})\cdot\phi_1  \right)= (\overline{a (D\phi_0;D\phi_1)_{(-i\theta)}})\cdot(\overline{(D\phi_1;D\phi_0)_{(-i\theta)}\phi_0})=a\phi_0,
$$
again by the uniqueness of analytic continuation, taking into account that  $(D\phi_0;D\phi_1)_{t}= ((D\phi_1;D\phi_0)_t)^*= ((D\phi_1;D\phi_0)_t)^{-1}$ for $t\in\R$.
\end{proof}

\section*{Acknowledgements}
We thank Mikael de la Salle and Gilles Pisier for their generous (and swift, of course) help during the first stages of this research, by 2007.

We are indebted to the referee who suggested the notion of a lazy centralizer appearing in the text, the form that the statement of Theorem~\ref{main} had to have, and how to prove it.

\end{document}